\documentclass[aihp,preprint]{imsart}

\RequirePackage[OT1]{fontenc}
\RequirePackage{amsthm,amsmath,amssymb,amsfonts,csquotes,mathrsfs}\RequirePackage{graphicx}
\RequirePackage{xr}
\RequirePackage[numbers]{natbib}
\RequirePackage[colorlinks,citecolor=blue,urlcolor=blue]{hyperref}

\startlocaldefs
\numberwithin{equation}{section}
\theoremstyle{plain}
\newtheorem{thm}{Theorem}[section]
\newtheorem{remark}[thm]{Remark}
\newtheorem{defn}[thm]{Definition}
\newtheorem{ppn}[thm]{Proposition}

\newtheorem{lem}[thm]{Lemma}

\newcommand{\E}{\mathbb{E}}

\newcommand{\N}{\mathbb{N}}

	\renewcommand{\P}{\mathbb{P}}

\newcommand{\R}{\mathbb{R}}

\newcommand{\Z}{\mathbb{Z}}

\newcommand{\cB}{\mathcal{B}}
\newcommand{\cC}{\mathcal{C}}

\newcommand{\cN}{\mathcal{N}}

\newcommand{\NN}{\mathcal{N}}

\def\dd{\mbox{d}}

\newcommand{\one}{{\bf 1}}

\newcommand{\os}[2]{\overset{\text{{\tiny #1}}}{#2}}

\def\fa{\textcolor{black}}

\endlocaldefs

\begin{document}

\begin{frontmatter}

\title{Persistence exponents in Markov chains}
\runtitle{Persistence exponents in Markov chains}
\thankstext{m1}{Research partially supported by DFG grant AU370/4-1.}
\thankstext{m2}{Research partially supported by NSF grant DMS-1712037.}
\thankstext{m3}{Research partially supported by grant 147/17 from the Israel Science Foundation and by a US-Israel BSF grant.}
\date{October 15, 2018}
\begin{aug}
\author{\fnms{Frank} \snm{Aurzada}\thanksref{m1}\ead[label=e1]{aurzada@mathematik.tu-darmstadt.de}},
\author{\fnms{Sumit} \snm{Mukherjee}\thanksref{m2}\ead[label=e2]{sm3949@columbia.edu}}
\and
\author{\fnms{Ofer} \snm{Zeitouni}\thanksref{m3}\ead[label=e3]{ofer.zeitouni@weizmann.ac.il }}
\runauthor{Aurzada, Mukherjee and Zeitouni}
 
\affiliation{Technical University of Darmstadt\thanksmark{m1}; Columbia University\thanksmark{m2};  Weizmann Institute of Science\thanksmark{m3}}

%
%

\address{Fachbereich Mathematik\\ Technische Universit\"at Darmstadt,\\ Darmstadt, Germany\\ E-mail: \href{mailto:aurzada@mathematik.tu-darmstadt.de }{aurzada@mathematik.tu-darmstadt.de }}

\address{Department of Statistics,\\ Columbia University,\\ New York, USA\\ E-mail: \href{mailto:sm3949@columbia.edu }{sm3949@columbia.edu } }

\address{Department of Mathematics,\\ Weizmann Institute of Science,\\ Rehovot, Israel \\ E-mail: \href{ofer.zeitouni@weizmann.ac.il }{ofer.zeitouni@weizmann.ac.il  } }

\end{aug}

\begin{abstract}
  We prove the existence of the \textit{persistence exponent}
$$\log\lambda:=\lim_{n\to\infty}\frac{1}{n}\log \P_\mu(X_0\in S,\ldots,X_n\in S)$$ 
for a class of time homogeneous Markov chains 
$\{X_i\}_{i\geq 0}$ taking values in a Polish space, where $S$ is a
Borel measurable set and $\mu$ is an initial distribution.
Focusing on the case of  AR($p$) and MA($q$) processes with $p,q\in \N$ and continuous innovation
distribution, we study the existence of $\lambda$ and its continuity in the parameters of the AR and MA processes, respectively, for $S=\R_{\geq 0}$. For AR processes with log-concave innovation distribution, 
  we prove the strict monotonicity of $\lambda$. Finally, we compute new explicit exponents in several concrete examples.


\end{abstract}



\end{frontmatter}

\section{Introduction}\label{sec:intro}

Let $\{X_i\}_{i\ge 0}$ be a time homogenous Markov chain on a Polish
space with transition kernel $P(x,\dd y)$.  For a given Borel measurable set $S$, we are interested in the asymptotics of the persistence probability 
$$ p_n(P,S,\mu):=\P_\mu(X_i\in S,0\le i\le n)=\int_{S^{n+1}}P(x_i,\dd x_{i+1})\mu(\dd x_0),$$ where $\mu$ is the initial distribution, i.e. the law of $X_0$. We stress that we shall be particularly interested in non-compact $S$. 
We will be interested in the existence of the 
  \textit{persistence exponent} $\lambda=\lambda(P,S,\mu)$, defined as
  \begin{equation}
    \log\lambda(P,S,\mu):=\lim_{n\to\infty}\frac{1}{n}\log p_n(P,S,\mu)
  \end{equation}
and its continuity and monotonicity properties in parameters of the kernel. 

The asymptotics of persistence probabilities 
for not necessarily Markov processes
has received both classical and recent interest in probability theory and theoretical physics. For recent surveys on persistence probabilities we refer the reader to \cite{BMS} for a theoretical physics point of view and to \cite{Aurzada_Simon} for a review of the mathematical literature. \fa{ We note in passing that outside the literature on Markov processes, a big component of the literature on persistence is that for Gaussian processes, with links in the stationary case to their spectral properties
 (see \cite{DM, DM2,FF,FFN,FFNNS} and references therein for some recent developments in this area).}

Our approach exploits the Markovian structure and relates
  the persistence exponent to an eigenvalue of an appropriate operator, via
  the Krein-Rutman theorem. Such ideas have been extensively employed to study general versions of 
  the persistence problem for Markov processes, 
  under the name of quasi-stationary distributions (see Tweedie  
\cite{Tweedie,Tweedie2}, and  for more recent work, see
e.g.\ \cite{CV,CMS,MV}). 
 We work under somewhat different assumptions than is typical 
 in that literature, for the sake of the applications that we have in mind, which are MA (moving average) processes and AR (auto regressive) processes of finite order (to be defined later).
  In particular, we do not assume that the operator is irreducible;
and much of our effort lies in deriving the existence of 
the persistence exponent and its properties directly in terms of the kernel. 
The quasi-stationary approach developed in 
\cite{Tweedie,Tweedie2}  shows, under assumptions that are not always satisfied
in the examples that we consider, the equivalence of the exponent's 
existence and  properties of the eigenvalue equation determined by $P_S$ (c.f. \eqref{eq-P_Sdef} for the definition of $P_S$).
One of our key observations is that, 
even in very natural examples as in Section
  \ref{sec:eg-proof}, we often
need to work not with $P_S$ but rather with a modification of it. In addition,
the existing literature is focused on persistence exponents for
Dirac initial conditions,
whereas we in general require
the initial distribution $\mu$ to charge 
all open sets. 
  If the operator $P_S$ is assumed to be irreducible, then all our 
  results apply to degenerate initial distributions. Even in the irreducible 
  case,
  the persistence 
  exponent need not exist for general initial distributions; see
  Proposition \ref{ppn:ar_not_strict} for an example where the persistence
  exponent exists and is universal if the initial distribution is an atom,
  but does not need to not exist for general initial distributions.  

\fa{A more detailed study of persistence for AR processes of order $1$
 is provided by the very recent paper \cite{HKW}. See also the works \cite{baumgartendissertation,baumgartenar}, where the author studies persistence of AR processes for general innovation distributions with orders $\{1,2\}$.  Of relevance is the recent work \cite{CV}, that gives general criteria for convergence in total variation norm to a quasi-stationary distribution, for continuous time Markov processes. We do note that in our general setup, we cannot verify their conditions. In general, it is not clear that the existence of the persistence exponent, even under our conditions, would imply the existence of a Yaglom limit, i.e. of a  quasi-limiting distributions. Also related to our persistence problem is the study of the population size in critical branching processes, see e.g.\ \cite{afanasev}, also see \cite{maillard} for a recent work concerning Yaglom limits in this context. }

{One upshot of our approach is a study of 
monotonicity and continuity properties of the persistence exponent in parameters of the kernel $P$. 
We illustrate this in the case of MA processes and AR processes, where the kernel (and thus the persistence exponent) depends on the coefficient vector. In the setting of AR processes, we derive a monotonicity lemma 
  (Lemma~\ref{lem:logconcave}) that might be of independent interest. }
    As an application, we prove \textit{strict} monotonicity of the
  persistence exponent for AR($p$) processes with log concave innovation distributions.
Finally, we demonstrate the strength of our approach by computing a number of new persistence exponents in concrete examples by solving the corresponding eigenvalue equation.

The outline of the paper is as follows: Section~\ref{sec:mainexistenceresults}
contains our main abstract existence result.
The short and technical Section~\ref{sec:mainmonotonicityresult} 
contains an abstract
monotonicity lemma and a continuity lemma.
The abstract framework is
then applied in Section~\ref{sec:eg} to 
moving-average (MA) processes and auto-regressive (AR), where existence of the exponent, continuity of the exponent, (strict) monotonicity results, and the question whether the exponent is degenerate are discussed. Finally, Section~\ref{sec:comp} contains a number of concrete cases where we are able to solve the eigenvalue equation, i.e.\ to find the leading eigenvalue explicitly.
Sections~\ref{sec:intro-proof}--\ref{sec:comp-proof} are devoted to the proofs corresponding to the former three topics, respectively.

\subsection{Existence of the exponent} \label{sec:mainexistenceresults}
We begin with a definition. Throughout, $\|g\|_\infty$ denotes the sup norm of
a function $g$ on $S$.


\begin{defn}\label{def:operator}
Let $\cB(S)$ denote the set of all bounded measurable 
functions on $S$, and let $\cC_b(S)\subset \cB(S)$ denote the 
space of continuous bounded functions on $S$ equipped with the sup norm.
For a bounded linear operator $K$ mapping $\cB(S)$ to itself, define
  the operator norm  $$\|K\|:=\sup_{g\in \cB(S):\|g\|_{\infty}
\le 1}\|Kg\|_\infty,$$
and  the spectral radius
\begin{equation}
  \label{specraddef}
  \lambda(K):=\lim_{n\rightarrow\infty}\|K^n\|^{1/n}.
\end{equation}
\end{defn}
\noindent
Note that the limit in \eqref{specraddef}
exists by sub-additivity, and that $\lambda(K)\le \|K\|$.
 Note also that $p_n(P,S,\mu)= \int_S P_S^n \one (x) \mu(\dd x)$, where
$P_S$ is the linear operator on $\cB(S)$ defined by
\begin{equation}
  \label{eq-P_Sdef}
  [P_S(g)](x):=\int_{S}g(y)P(x,{\rm d} y),\qquad x\in S, 
\end{equation}
while, for comparison, the spectral radius satisfies
$$
\lambda(P,S) :=\lambda(P_S)= \lim_{n\to\infty} \left( \sup_{x\in S} P_S^n \one (x) \right)^{1/n}.
$$

%

%
%

We recall that an operator $K$ from $\cC_b(S)$ to itself is called compact if for any sequence $\{g_n\}_{n\geq 1}$ in $\cC_b(S)$ with $\|g_n\|_{\infty}
\leq 1$ one finds a subsequence $\{n_k\}_{k\ge 1}$ such that $\{K g_{n_k}\}_{k\ge 1}$ converges in sup norm.

\begin{thm}\label{thm:exponent}
Assume the following conditions:
\begin{enumerate}
\item[(i)]
$K$ is a non-negative linear operator which maps $\cC_b(S)$ into itself, and 
$K^k$ is compact for some $k\geq 1$.
\item[(ii)] 
 $\mu$ is a probability measure such that
 $\mu(U)>0$ for any non empty open set $U\subseteq S$. 
\end{enumerate}
Then,
\begin{align}\label{eq:exponent1}
\lim_{n\rightarrow\infty}\frac{1}{n}\log \left(\int_S K^n{\bf 1}(x)\mu(\dd x)\right)= \log\lambda(K).
\end{align}
Further, if $\lambda(K)>0$, then $\lambda(K)$ is 
the largest eigenvalue of the operator $K$, the corresponding 
eigenfunction  $\psi\in \cC_b(S)$ is non-negative,
and there exists
a bounded, non-negative, finitely additive regular
measure $m$ on $S$  which is a left eigenvector of $K$ corresponding to the eigenvalue $\lambda(K)$, i.e. 
$$\int_Sm({\rm d}x)\int_S K(x,{\rm d} y)f(y)=\lambda(K)m(f), \quad
f\in \cC_b(S).$$
\end{thm}
\begin{remark}
  a) Replacing $K$ by $P_S$ in Theorem \ref{thm:exponent} yields 
  a sufficient condition for the existence of a universal  persistence exponent for all initial conditions $\mu$ satisfying condition (ii). As we will see in Section \ref{sec:eg}, this is not always the best choice.\\
 b) The assumption of compactness of $K^k$ for some $k$ (rather than the compactness of $K$ itself) is, on the one hand, sufficient for the proof to go through and, on the other hand, necessary for dealing with some concrete examples. For example, the operator $P_S$ related to MA($q$) processes is typically not compact, whereas $P_S^{q+1}$ is compact.\\
c) The left eigenvector $m$ in Theorem \ref{thm:exponent} 
is only finitely additive.
This is a consequence of the fact that $S$ can be (and typically is, in our applications) non-compact. This complicates some of the following arguments. For example, 
the proof of Proposition \ref{ppn:example_ma_exponential_allq} would be immediate if $m$ were a measure.
\end{remark}


%

\subsection{Properties of exponents} \label{sec:mainmonotonicityresult}



We begin with a definition.
 \begin{defn}

Suppose $S$  is equipped with a partial order $\le_S$. 
Let $\cB_{+,>}(S)$ denote the class of bounded, non-negative, non-decreasing (in the sense of this partial order) measurable functions on $S$.

 A non-negative bounded linear operator $K$ on $\cB(S)$ is said to be
 non-decreasing  with respect to the partial order $\le_S$, if $K$ maps
 $\cB_{+,>}(S)$ to itself.
\end{defn}
The following lemma gives a sufficient condition for comparing $\lambda(K_1)$ and $\lambda(K_2)$ for two bounded non-negative linear operators  $K_1,K_2$. 


\begin{lem}\label{lem:monotone}
Let $K_1$ and $K_2$ be two bounded non-negative linear operators  on $\cB(S)$, such that the following conditions hold:
\begin{enumerate}
\item[(i)]
There exists a non-negative measurable function $h$ on $S$ such that 
$[K_1(g)](x)\ge h(x)[K_2(g)](x)$ for any $x\in S, g\in \cB_{+,>}(S)$.

\item[(ii)]
$K_1$ is  non-decreasing on $S$.
\end{enumerate}

Then for any $g\in \cB_{+,>}(S)$ we have
$K_1^n(g)\ge K_{2,h}^n(g),$
where $[K_{2,h}(g)](x):=h(x)[K_2(g)](x)$.

\end{lem}

The next lemma, relating the continuity 
of exponents to continuity in operator norm, 
is useful when studying the continuity of exponents.

\begin{lem}\label{lem:conti}
For every $1\le \ell\le \infty$ let $K_\ell$ be a bounded linear operator on $\cB(S)$. If $\lim_{\ell\rightarrow\infty}\|K_\ell-K_\infty \|=0$, then $\lim_{\ell\to\infty}\lambda(K_\ell)=\lambda(K_\infty)$.

\end{lem}

\section{Results for MA and AR processes}\label{sec:eg}

In this section we consider our two main examples, 
 moving-average processes and auto regressive processes.

\subsection{Moving Average processes}\label{sec:ma}
Let $\{\xi_i\}_{i\ge -q}$ be a sequence of i.i.d.\ random variables from a continuous distribution function $F$. For a coefficient vector ${\bf a}:=(a_1,\ldots,a_q)\in \R^q$ define the moving average
(MA($q$)) process $\{Z_i\}_{i\ge 0}$ by setting
$$
Z_i:=\xi_i+\sum_{j=1}^q a_j\xi_{i-j},\qquad i=0,1,2,\ldots.
$$ 
Define the operator
$K$  mapping $\cC_b(\R^q)$ to itself by
\begin{equation}
  \label{eq-KMA}
  K g(x_1,\ldots,x_q)=\int_{y+\sum_{j=1}^q a_j x_{q+1-j}>0} 
  g(x_2,\cdots,x_q,y) F({\rm d}y).
\end{equation}

\begin{thm}\label{thm:mainfty} 
For all MA($q$) processes with $\P(Z_0\geq 0)<1$, 
there is a $\beta_F({\bf a})\in [0,1)$ so that
%
\begin{equation}
  \label{eq-1pg9}
\lim_{n\rightarrow\infty}  \frac{1}{n}\log \P(\min_{0\le i\le n}Z_i
  \geq 0)=\log \beta_F({\bf a}).
\end{equation}
Further, 
 if $\beta_F({\bf a})>0$ then $\beta_F({\bf a})$ is the largest 
 eigenvalue of the operator  $K$ defined in \eqref{eq-KMA},
 and the corresponding eigenfunction $\psi(\cdot)$ is non-negative and continuous.
\end{thm}


The next theorem establishes the
continuity of the MA($q$) persistence exponent.
\begin{thm}\label{thm:ma_conti}
In the setting of Theorem \ref{thm:mainfty}, the function ${\bf a}\mapsto \beta_F({\bf a})$ is continuous on $\R^q$.
\end{thm}

Theorem \ref{thm:mainfty} shows that $\beta_F({\bf a})\in [0,1)$. 
  As noted in \cite{krishnapur,majumdardhar2001},
  for the particular case $q=1$ and $a_1=-1$ and any innovation distribution with a continuous density, we have  
\begin{equation} \label{eqn:krishnapur}
  \P(\min_{0\le i\le n}Z_i\ge 0)=\P(\xi_{-1}<\xi_0<\ldots<\xi_n)=\frac{1}{(n+2)!},
\end{equation}
and so $\beta_F(-1)=0$. The next proposition gives a necessary and sufficient condition for $\beta_F({\bf a})>0$. 
\begin{ppn}\label{thm:mafinite}
  Suppose that $\{Z_i\}_{i\geq 0}$ is a
   MA($q$) process such that $\P(\xi_1>0)>0,\P(\xi_1<0)>0$.
Then $\beta_F({\bf a})>0$ if and only if $\sum_{j=1}^q a_j\ne -1$.
\end{ppn}



\subsection{Auto-regressive processes}\label{sec:auto}
Let $\{\xi_i\}_{i\ge p}$ be a sequence of i.i.d.\ random variables 
of law $F$ possessing a density function $\phi(\cdot)$ with respect to the Lebesgue measure. Let ${\bf a}:=(a_1,\ldots, a_p)\in \R^p$ be a vector, called coefficient vector.
Given ${\bf Z}_0:=(Z_0,\ldots, Z_{p-1})\in \R^p$ independent of the sequence $\{\xi_i\}$, we define an AR($p$) process $\{Z_i\}_{i\ge p}$ by setting
$$
Z_{i}:=\sum_{j=1}^p a_j Z_{i-j}+\xi_i,\qquad i\ge p.
$$
The law of 
${\bf Z}_0$ is denoted by $\mu$ and called the \textit{initial distribution}. 
Furthermore, let $K:\cB([0,\infty)^p)\to \cB([0,\infty)^p)$ be defined by
  \begin{equation}
    \label{def-KAR}
    K\psi(x_1,\cdots,x_p):=\int_{y+\sum_{j=1}^pa_j x_{p+1-j}>0}\psi(x_2,\cdots,x_p,y+\sum_{j=1}^pa_jx_{p+1-j})\phi(y){\rm d}y. 
  \end{equation}
 Under the above assumptions, $K$ maps
$\cC_b\Big([0,\infty)^p\Big)$ to itself.

\fa{The behavior of the persistence probabilities of AR
processes is surprisingly rich. The recent work \cite{demboar} studies AR processes with a Gaussian innovation density, uncovering a rich structure of varying rates of persistence decay from exponential to stretched exponential, polynomial and constant. 
On the contrary, our approach is mostly equipped to handle the case when the persistence decay is exponential, so we concentrate on a part of this spectrum. One advantage of our approach is that it also gives the existence of a persistence exponent, which, to the best of our understanding, does not follow from \cite{demboar}.}

In this paper we treat two particular sub parameter regimes while studying persistence exponents for the AR($p$) process, namely
${\bf a} \leq {\bf 0}$  and \{${\bf a}\in \R^p: \{\sum_{i=1}^p |a_i| < 1\}$. The behavior of the operator is somewhat different in these two regimes, so we have to treat them separately.
\begin{thm}\label{thm:ar} 
Fix $p\in \N$,
  ${\bf a}\le {\bf 0}$, an innovation density $\phi(\cdot)$,
  initial
distribution $\mu$ satisfying $\mu(U)>0$ for all open $U\in \R_+^p$, and
let $\{Z_i\}_{i\geq 0}$ be the associated AR($p$) process.

\begin{enumerate}
\item[(a)]
  There is a $\theta_F({\bf a})\in[0,1]$, independent of $\mu$,
  such that 
  $$\lim_{n\rightarrow\infty}\frac{1}{n}\log \P_\mu(\min_{0\le i\le n}Z_i\geq 
  0)=\log \theta_F({\bf a}).$$ Further, if $\theta_F({\bf a})>0$, then $\theta_F({\bf a})$ is the largest eigenvalue 
 of the operator $K$ from \eqref{def-KAR}, viewed as an operator mapping
 $\cC_b\Big([0,\infty)^p\Big)$ to itself.
   The corresponding eigenvector $\psi$ is non-negative and continuous. 

\item[(b)]
If $\P_F(\xi_1>0)>0$ then $\theta_F({\bf a})>0$, and if $\P_F(\xi_1>0)<1$ then $\theta_F({\bf a})<1$.

\item[(c)]
If ${\bf a}^{(k)}$ is a sequence of vectors in $(-\infty,0]^p$ converging to ${\bf a}$ and $a_p<0$, then  $\lim_{k\rightarrow\infty}\theta_F({\bf a}^{(k)})=\theta_F({\bf a})$.

\end{enumerate}
\end{thm}

 As the next proposition shows, the persistence exponent may not exist for 
 some initial distributions, if the coefficient vector 
 ${\bf a}$  in Theorem~\ref{thm:ar} is allowed to have positive entries.

\begin{ppn}\label{ppn:ar_not_strict}
  { Suppose $\{Z_i\}$ is an   AR($1$) process with 
  innovation distribution $F=\NN(0,1)$, and $a_1\in [0,1)$.}

\begin{enumerate}
\item[(a)]
 {If the initial distribution is $\NN:=\NN(0,1/(1-a_1^2))$, then the exponent $\theta_F(a_1,\NN)$ defined by 
 $$\log \theta_F(a_1,\NN):=\lim_{n\rightarrow\infty}\frac{1}{n}\log \P(\min_{0\le i\le n}Z_i\ge 0)$$
 exists, belongs to $(0,1)$, 
 and is continuous as a function of $a_1$.}

\item[(b)]
 If the initial distribution $\mu$ satisfies
\begin{align}
\limsup_{M\rightarrow\infty}\frac{1}{\log M}\log \P_\mu(Z_0>M)=&0\label{eq:condition_1},\\
\liminf_{M\rightarrow\infty}\frac{1}{\log M}\log \P_\mu(Z_0>M)=-&\infty\label{eq:condition_2},
\end{align}
then there exist sequences $\{m_k\}_{k\ge 1}$ and $\{n_k\}_{k\ge 1}$ such that
\begin{align}
\lim_{k\rightarrow\infty}\frac{1}{m_k}\log \P(\min_{0\le i\le m_k}Z_i\ge 0)= &0\label{eq:no_limit_1},\\
\lim_{k\rightarrow\infty}\frac{1}{n_k}\log \P(\min_{0\le i\le n_k}Z_i\ge 0)= &\log \theta_{F}(a_1,\NN),\label{eq:no_limit_2}
\end{align}
where $\theta_{F}(a_1,\NN)$ is as in part (a). In particular, the limit 
$$\lim_{n\rightarrow\infty}\frac{1}{n}\log \P(\min_{0\le i\le n}Z_i\ge 0)$$ 
does not exist for any $a_1\in (0,1)$.

\item[(c)]
If $\mu=
\delta_{x_0}$  for some $x_0>0$, then the exponent $\theta_F(a_1,\delta_{x_0})$ defined by $$\log\theta_F(a,\delta_{x_0}):=\lim_{n\rightarrow\infty}\frac{1}{n}\log \P(\min_{0\le i\le n}Z_i\ge 0)$$ exists, and equals $\theta_F(a_1,\NN)$ of part (a). 
\end{enumerate}
\end{ppn}

It follows that the AR($1$) operator $K=P_S$ with $a_1\in (0,1)$ and $F=\NN(0,1)$  is no longer compact on $\cC_b\Big([0,\infty)\Big)$, as otherwise Theorem~\ref{thm:exponent} would be applicable with $K=P_S$, giving the existence of an exponent in part (b). On the other hand,
there does exist a universal exponent for initial distributions consisting of a single atom.
This motivates our focus in this paper on
studying nonatomic initial distributions. \fa{It may be possible to weaken the assumption that the initial distribution puts mass on all open sets. In particular, one may be able to adapt the technique used in the recent work \cite{CV}, and work with a weighted sup norm as opposed to the usual sup norm considered in this paper. It is however not clear whether the Harnack condition used in \cite{CV} holds for AR processes, particularly without extra assumptions on the innovation density.}

In order to derive an existence result for
situations where the operator $P_S$ is not compact, 
one needs to make a judicious choice of the operator 
$K$ in Theorem \ref{thm:exponent}. This requires additional assumptions on
the initial measure  and innovation. We focus below on the contractive case
$\sum_{j=1}^p|a_j|<1$.


\begin{thm}\label{thm:ar+}
  Fix $p\in \N$, parameters ${\bf a}$ satisfying $\sum_{j=1}^p|a_j|<1$,
 an innovation density $\phi(\cdot)$, an initial
distribution $\mu$ satisfying $\mu(U)>0$ for all open $U\in \R_+^p$, and
let $\{Z_i\}_{i\geq 0}$ be the associated AR($p$) process.
Further assume that there exists $\delta>0$ such that 
\begin{equation}
  \label{eq-expmomass}
  \E_\mu\Big[\exp\Big\{\delta\sum_{j=0}^{p-1}Z_j\Big\}1_{\{\min_{0\le i\le p-1}Z_i
  \geq 0\}}\Big]<\infty
\end{equation}
and 
\begin{align}\label{eq:+}
\limsup_{|t|\rightarrow\infty}\frac{1}{|t|}\log\phi(t)<0.
\end{align}

\begin{enumerate}
\item[(a)]
 There is a $\theta_F({\bf a})\in[0,1]$, independent
 of $\mu$, such that $$\lim_{n\rightarrow\infty}\frac{1}{n}\log \P_\mu(\min_{0\le i\le n}Z_i \geq 0)=\log \theta_F({\bf a}).$$ Further, if $\theta_F({\bf a})>0$, then $\theta_F({\bf a})$ is an eigenvalue of the operator $K$ on
 $\cC_b\Big([0,\infty)^p\Big)$ defined by \eqref{def-KAR}.
 The corresponding eigenfunction $\psi$ is non-negative and continuous. 

\item[(b)]
If $\P_F(\xi_1>0)>0$ then $\theta_F({\bf a})>0$, and if $\P_F(\xi_1>0)<1$
then $\theta_F({\bf a})<1$. 

\item[(c)]
The function ${\bf a}\mapsto \theta_F({\bf a})$ is continuous on the set $\sum_{j=1}^p|a_j|<1$. 

\end{enumerate}

\end{thm}

As mentioned before, the proof of
Theorem \ref{thm:ar+} employs 
a modified version of the operator $P_S$, which now turns out 
to be compact if $\sum_{j=1}^p|a_j|<1$. 
The motivation behind the modification of the operator borrows from 
\cite{aurzadabaumgartenalea,baumgartendissertation,majumdar_bray_ehrhardt_2001}, where a similar strategy was used to deal with  AR($1$) processes with {\em Gaussian} innovations starting at stationarity. An equivalent proof of Theorem \ref{thm:ar+} might be obtained by replacing the sup norm topology on $\cC_b([0,\infty)^p)$ by 
a weighted sup norm with geometrically growing weights, 
which ensures that $P_S$ 
is compact with respect to this new topology.

 \fa{We also note that the assumption of existence of a finite exponential moment for the initial distribution is only needed for the upper bound, in the sense that even if \eqref{eq-expmomass} is removed, one can still show that
$$\liminf_{n\rightarrow\infty}\frac{1}{n}\log \P(\min_{0\le i\le n}Z_i>0)\ge \log \theta_F({\bf a}),$$
where $\theta_F({\bf a})$ is the universal exponent from Theorem \ref{thm:ar+}. Only the upper bound may fail here without \eqref{eq-expmomass}, as shown in Proposition \ref{ppn:ar_not_strict}.}


\subsubsection{Strict monotonicity of the exponent}

If we restrict ${\bf a}$ to the non negative orthant $[0,\infty)^p$, a simple coupling argument shows that for any initial distribution $\mu$ the function ${\bf a}\mapsto \P^{\bf a}(\min_{0\le i\le n}Z_i\ge 0)$ is monotonically non-decreasing in ${\bf a}$. This implies that the exponent $\theta_F({\bf a},\mu)$ with initial distribution $\mu$ is non-decreasing in ${\bf a}\in [0,\infty)^p$, provided it exists.
 Note however that if in Proposition \ref{ppn:ar_not_strict}, the limit in \eqref{eq:condition_2} equals $0$, then the same proof shows that the corresponding exponent $\theta_{F}(a_1,\mu)$ equals $0$ for all $a_1\in (0,1)$, and consequently the function $a_1\mapsto \theta_F(a_1,\mu)$ is not strictly monotone.

 Our next theorem shows that if $F$ has a log concave density on
 $\R$ and the initial distribution $\mu$ has finite exponential moment, then the exponent $\theta_F({\bf a},\mu)$ is free of $\mu$, and the map ${\bf a}\mapsto \theta_F({\bf a})$  is \textit{strictly}
 increasing on the set $\{{\bf a}\ge {\bf 0}:\sum_{j=1}^p a_j<1\}$. 
 The exponential decay of log concave densities ensures that \eqref{eq:+} holds, and so Theorem \ref{thm:ar+} guarantees the existence of a non-trivial exponent which is  free of the initial distribution $\mu$.
\begin{thm}\label{thm:strict_ar}
Assume that $\phi$
is a strictly positive log concave density  over $\R$, that
${\bf a}\ge {\bf 0}, \sum_{j=1}^p a_j<1$, and that
$\mu$ satisfies \eqref{eq-expmomass}.
Then ${\bf b}\ge {\bf a}$ with ${\bf b}\ne {\bf a}$ implies $\theta_F({\bf b})> \theta_F({\bf a})$.
\end{thm}

We complete the picture on the positive orthant through the next proposition, 
which states that the persistence exponent $\theta_F({\bf a})=1$ for all ${\bf a}\ge {\bf 0}$ such that $\sum_{j=1}^p a_j>1$, for any innovation distribution $F$. 

\begin{ppn}\label{thm:non_strict}
  Assume that  ${\bf a}\ge 0$ and $\sum_{j=1}^p a_j> 1$.  If the initial distribution satisfies $\mu((0,\infty)^p)>0$, and  the innovation density satisfies $\P(\xi_1>0)>0$,  then  
  $\theta_F({\bf a},\mu)=
  \theta_F({\bf a})=1$.
\end{ppn}

Proposition \ref{thm:non_strict}, together with Theorems \ref{thm:ar+} and \ref{thm:strict_ar}, gives an almost complete picture in terms of monotonicity on the positive orthant.
The function ${\bf a}\mapsto \theta_F({\bf a},\mu)$ is continuous and non-decreasing on $\{{\bf a}:\sum_{j=1}^p a_j<1\}$, and identically equal to $1$ on the set $\{{\bf a}:\sum_{j=1}^p a_j> 1\}$. If further the innovation density is log concave and the initial distribution has finite exponential moment, then the exponent is {\it strictly} increasing on $\{{\bf a}:\sum_{j=1}^p a_j<1\}$. In the critical case, the exponent is usually one, as shown in some specific examples in \cite{baumgartenar,demboar}.

\subsubsection{Positivity of the exponent}
Part (b) of Theorem~\ref{thm:ar} and Theorem \ref{thm:ar+} 
give conditions ensuring that
the exponent is non-trivial, i.e. the persistence probability decays at an exponential rate. The next proposition generalizes this to show that no matter what the coefficient vector ${\bf a}$ may be, the exponent can never be $0$, i.e. the persistence probability can never decay at a super exponential rate. 

\begin{ppn}\label{ppn:finite}
  Fix $p\in \N$, parameters ${\bf a}$, an innovation distribution such that $0$ is an interior point of its support,
  and $\mu$ satisfying 
  $\mu(( 0,\delta)^p) >0$ for every $\delta>0$. 
Let $\{Z_i\}_{i\geq 0}$ be the associated AR($p$) process.
Then,
$$\liminf_{n\rightarrow\infty}\frac{1}{n}\log\P(\min_{0\le i\le n}Z_i
\geq 0)>-\infty.$$
In particular, 
if $\theta_F({\bf a},\mu)$ exists then it must be positive.

\end{ppn}

\begin{remark}
 One cannot dispense completely of the assumptions in Proposition \ref{ppn:finite}. Indeed, concerning the condition on initial distribution,
 when $p=1$, $a_1=-\frac{1}{2}$, $\mu(\left( 2,4) \right))=1$
 and $\P_F(\left( 0,1 \right))=1$, one sees that $Z_0\ge 2$ forces $Z_1=-\frac{1}{2}Z_0+\xi_1<0$, and so $\theta_F(a_1,\mu)=0$. 
 On the other hand, concerning the condition on the innovation distribution,
 if $p=1$, $a_1=1$, $\P_F(\left( -1,-2 \right))=1$ 
 and $\mu(\left(x,\infty  \right))=e^{-x^2}$ for all $x>0$, one obtains that
 $\P(\min_{0\le i\le n}Z_i\geq 0)\le \P(Z_0
 \geq n)=e^{-n^2}$, and so again
 $\theta_F(a_1,\mu)=0$. 
 
 \end{remark}

\section{Exponents for concrete cases}\label{sec:comp}
Using our operator approach, we can compute the persistence exponent in a number of concrete examples. 

\subsection{MA($1$) processes}\label{sec:ma-concrete}

We first consider MA(1) processes, starting with uniform innovation density.

\begin{ppn} \label{ppn:example_ma_uniform}
  Let $\{Z_i\}_{i\ge 0}$ be a MA(1) process with $a_1=1$
and innovation density $\phi=$ $(a+b)^{-1} 
{\bf 1}_{(-a,b)}$, where $a,b>0$.
\begin{itemize}
\item
  If $a\ge b$ then
$$
\P(\min_{0\le i\le n}Z_i\geq 0)  = \Big(\frac{4b}{\pi(a+b)}\Big)^{n+o(n)}.
$$
\item
  If $a<b$ then
$$
\P(\min_{0\le i\le n}Z_i \geq 0) = \lambda^{n+o(n)}.
$$
where $\lambda$
is the largest real solution to the equation
  \begin{equation} \label{eqn:uniformmasoln}
  \tan \left( \frac{a}{(a+b)\lambda} \right) = \frac{1- (1-2a/(a+b)) /\lambda}{1+(1-2a/(a+b)) /\lambda}.
  \end{equation}  
 \end{itemize}
\end{ppn}



For $a=b$ in Proposition \ref{ppn:example_ma_uniform}, one obtains $\beta_F(1)=2/\pi$. The next theorem shows that for continuous 
{\it symmetric innovation distributions} this value is {\it universal}. 

\begin{thm} \label{thm:example_ma_general_at_rho1}
 Let $\{Z_i\}_{i\ge 0}$ be a MA(1) process with $a_1=1$
and symmetric innovation density. Then
$$
  \P(\min_{1\le i\le n}Z_i\geq 0) = \!\P(\min_{0\le i\le n}\xi_{i}+\xi_{i-1}\geq 0)=\!
\sum_{k\in\Z} \frac{2}{(\pi/2+2\pi k)^{n+2}}=\!\left(\frac{2}{\pi}\right)^{n+o(n)}\!\!\!.
$$
\end{thm} 
\noindent
Theorem \ref{thm:example_ma_general_at_rho1} first appears in
\cite{majumdardhar2001}, where the proof technique is  different.

Proposition~\ref{ppn:Rademacher} below shows that
the universality in Theorem \ref{thm:example_ma_general_at_rho1} does not extend to discrete distributions.
In fact, for discrete innovation distributions $F$,
there can be 
\emph{non-trivial} differences between the two quantities
$$
\P(\min_{0\le i\le n}Z_i>0 ) \qquad\text{and}\qquad \P(\min_{0\le i\le n}Z_i\ge 0).
$$
\fa{Indeed, we have the following proposition, whose proof is elementary and left to the reader.}

\begin{ppn}\label{ppn:Rademacher}
   Let $\{Z_i\}_{i\ge 0}$ denote an MA(1) process with $a_1=1$
and Rademacher innovations, i.e. 
$\xi_i$ equal
$\pm 1$ with probability $1/2$. Then 
$$
\P(\min_{0\le i\le n}Z_i>0 ) = (1/2)^{n+2},
$$
while
$$
\P( \min_{0\le i\le n}Z_i \ge 0)  = \left(\frac{1}{2}+\frac{1}{\sqrt{5}}\right) \left(\frac{1+\sqrt{5}}{4}\right)^{n+1} + \left(\frac{1}{2}-\frac{1}{\sqrt{5}}\right) \left(\frac{1-\sqrt{5}}{4}\right)^{n+1}.
$$
\end{ppn}

%

Our final MA example considers MA(1) processes
with exponential innovation distribution. 

%

\begin{ppn} \label{ppn:example_ma_exponential_allq}
   Let $\{Z_i\}_{i\ge 0}$ denote an MA(1) process with $a_1\in (-1,0)$ 
and standard exponential
innovations. Then
$$
\P(\min_{0\le i\le n}Z_i\geq 0)=(1 + a_1 )^{n+o(n)}.
$$
\end{ppn}

\subsection{AR(1) processes}\label{sec:ar-concrete}
We now consider persistence exponent for AR(1) 
processes with uniformly distributed innovations.

\begin{ppn}\label{ppn:ar_unif}
Let $\{Z_i\}_{i\ge 0}$ be an AR(1) process with $a_1=-1$, arbitrary
initial distribution $\mu$, and with innovation density $\phi=$ $(a+b)^{-1} 
{\bf 1}_{(-a,b)}$, where $a,b>0$.
Then 
$$
\P(\min_{0\le i\le n}Z_i\geq 0) =\left(\frac{2b}{\pi(a+b)}\right)^{n+o(n)}.
$$
\end{ppn}


Our final example concerns exponential innovations.

\begin{ppn}\label{ppn:ar_exponential}
  Let $\{Z_i\}_{i\ge 0}$ be an AR(1) process with $a_1<0$, arbitrary 
initial distribution $\mu$, and standard exponential innovations.
Then 
$$
\P(\min_{0\le i\le n}Z_i\geq 0) =
\left(\frac{1}{1-a_1}\right)^{n-1} \E e^{a_1Z_0}1_{\{Z_0 \geq 0\}}.
$$
\end{ppn}



\section{Proof of the results of Section \ref{sec:intro}}\label{sec:intro-proof}

 \begin{proof}[Proof of Theorem \ref{thm:exponent}]
The upper bound is simple: using that 
  ${\bf 1}(\cdot)\in \cC_b(S)$, we obtain from \eqref{specraddef} that
$$\int_S [K^n({\bf 1})](x)\mu(\dd x)\le 
\sup_{x\in S} [K^n \one] (x) = \|K^n (\one)\|_\infty\leq \lambda(K)^{n+o(n)}.$$

We turn to the lower bound. We may and will assume
that $\lambda:=\lambda(K)>0$ since otherwise there is nothing left to prove.
%
Note that $\cC_b(S)$ equipped with the sup norm $\|g\|_\infty:=\sup_{x\in S}|g(x)|$ is a Banach space (even if $S$ is not compact, see \cite{dunfordschwartz}, p.\ 257). Thus denoting by $K^k:\cC_b(S)\mapsto \cC_b(S)$ the $k$-fold composition of $K$ (note that we consider $K^k$ acting 
on the smaller space $\cC_b(S)$), by assumption (i), $K^k$ is a compact operator.
Further, 
$$\lim_{n\to \infty} \big( \|(K^k)^n\| \big)^{1/n} = \left(\lim_{n\to\infty}\Big(\|K^{nk}({\bf 1})\|_\infty\Big)^{\frac{1}{kn}}\right)^k=\lambda(K)^k>0,
$$
and so an application of  the Krein-Rutman theorem  (see \cite[Theorem 19.2]{Deimling} and \cite[Problem 7.1.9]{abramovich}) yields
the existence of a non-negative continuous function $\tilde{\psi}\in \cC_b(S)$, $\tilde{\psi}\neq 0$, such that $$K^k\tilde{\psi}(x)=\lambda^k \tilde{\psi}(x),\qquad \forall x\in S.$$ 
Setting $$\psi(x):=\sum_{a=0}^{k-1}\lambda^a [K^{k-1-a}(\tilde{\psi})](x),$$
we note that $\psi\in \cC_b(S)$. Also note that $\psi(x)\ge \lambda^{k-1}\tilde{\psi}(x)$, and so $\psi$ is non-zero and non-negative. Finally, a telescopic cancellation gives
$$K\psi-\lambda \psi=\sum_{a=0}^{k-1}\lambda^a K^{k-a}(\tilde{\psi})-\sum_{a=0}^{k-1}\lambda^{a+1} K^{k-1-a}(\tilde{\psi})=K^k\tilde{\psi}-\lambda^k\tilde{\psi}=0,$$
and so $K\psi=\lambda \psi$. Thus, setting $c:=\|\psi\|_\infty>0$, we 
obtain (using that $K$ preserves the order)
$$ [K^n({\bf 1})](x)\ge \frac{1}{c}[K^n(\psi)](x)=\frac{1}{c}\lambda^n\psi(x). $$
Integrating the last inequality with respect to $\mu$ gives
$$
\int_S [K^n({\bf 1})](x)\mu(\dd x)\ge \frac{\int_S \psi(x)\mu(\dd x)}{c} \lambda^n.
$$
Since $\int_S\psi(x)\mu(\dd x)>0$ by assumption (ii) on $\mu$,
  the lower bound in \eqref{eq:exponent1} follows at once.

  Finally, the fact that $\lambda=\lambda(K)$ is the largest eigenvalue of $K$ follows from the fact that $\lambda^k$ is the largest eigenvalue of $K^k$, another consequence of the Krein-Rutman theorem. Also, existence of the left eigenvector $m$ follows from \cite[Exercise 12, p.\ 236]{Deimling}, along with the observation that the dual of $\cC_b(S)$ is the space of bounded,    finitely additive regular measures on $S$, see \cite[Theorem IV.6.2.2]{dunfordschwartz}.
\end{proof}

  \begin{proof}[Proof of Lemma \ref{lem:monotone}]
  Since $g\in \cB_{+,>}(S)$, using assumption (ii)  we have $K_1^i(g)\in \cB_{+,>}(S)$ for all $1\le i\le n-1$. Using condition (i) 
  $K_1(g)\ge h(x)K_2(g)=K_{2,h}(g)$, which is the desired conclusion for $i=1$. To verify the statement for general $i$, we  proceed by induction:
  \begin{align*}
  K_1^i(g) & =K_1( K_1^{i-1} (g) ) \ge h(x) K_2 ( K_1^{i-1}(g))=K_{2,h}(K_1^{i-1}(g))\ge K_{2,h}^i(g).
  \end{align*}
  In the last display,
  we use the fact that $K_1^{i-1}(g)\in \cB_{+,>}(S)$ along with condition (i) for the first inequality, and the induction hypothesis along with the fact that $K_{2,h}$ preserves the ordering in the second inequality, which is true of any non negative operator.
  \end{proof}

\begin{proof}[Proof of Lemma \ref{lem:conti}]
Since $\|K_\ell-K_\infty\|$ converges to $0$, w.l.o.g.\ assume $\|K_\ell-K_\infty\|\le1$. Also w.l.o.g.\ by scaling all operators involved if necessary, we can assume that $\|K_\infty\|\le 1$. Thus, for any $f\in \cB(S)$ with $\|f\|_\infty\le 1$ and $\delta\in (0,1/2)$ arbitrary we have
\begin{align*}
\|K_\ell^nf\|_\infty=&\|(K_\infty+K_\ell-K_\infty)^nf\|_\infty\\
\le &\lfloor n\delta \rfloor {n\choose \lfloor n\delta\rfloor}\|K_\infty^{n-\lfloor n\delta\rfloor}\|+2^n\|K_\ell-K_\infty\|^{\lfloor n\delta\rfloor}.
\end{align*}
On taking sup over $f$, invoking \eqref{specraddef}  along with Stirling's approximation gives
$$
\lambda(K_\ell)\le \max\Big( \delta^{-\delta} (1-\delta)^{1-\delta}\lambda(K_\infty)^{1-\delta}, 2 \|K_\ell-K_\infty\|^{\delta}\Big).
$$
Letting $\ell\rightarrow\infty$ followed by $\delta\rightarrow 0$ gives
$$\limsup_{\ell\rightarrow\infty}\lambda(K_\ell)\le \lambda(K_\infty),$$
which is the upper bound. The lower bound follows by a symmetric argument, reversing the roles of $K_\ell$ and $K_\infty$.
\end{proof}

 \section{Proofs of the results of Section \ref{sec:eg}}\label{sec:eg-proof}

 \subsection{Proof of results in Subsection \ref{sec:ma}}\label{sec:ma-proof} 
  \begin{proof}[Proof of theorem \ref{thm:mainfty}]
The MA($q$) process is $q$-dependent, and so with $m=\lfloor\frac{n}{q+1}\rfloor$ we have
$$\P(\min_{0\le i\le n}Z_i\geq 0)\le \P(\min_{0\le i\le m}Z_{i(q+1)}
\geq 0)=\P(Z_0\geq 0)^{m+1},$$ from which 
$\beta_F({\bf a})<1$ follows.

The sequence $\{Z_i\}_{i\ge 0}$ is well defined and stationary. We now show existence of the exponent using Theorem \ref{thm:exponent} with $k=q+1$. 
Setting ${\bf X}(i):=(\xi_{i-q},\ldots,\xi_{i})$ we have that $\{{\bf X}(i)\}_{i\ge 0}$ is a time homogenous Markov chain on $\R^{q+1}$. Thus, with $$S:=\{{\bf x}:x_{q+1}+\sum_{j=1}^q a_j x_{q+1-j}>0\},$$  
the $q+1$ fold operator $P_S^{q+1}$ is given by
\begin{align*}
&[P_S^{q+1}(g)](x_1,\ldots,x_{q+1})\\
&\qquad=\int_{\R^{q+1}} g(x_{q+2},\ldots,x_{2q+2})\prod_{\ell=q+2}^{2q+2} {\bf 1}_{x_{\ell}+ \sum_{j=1}^qa_jx_{\ell-j}>0}  \dd F(x_{\ell}),\end{align*}
where $F$ is the distribution function of the innovation distribution. 
Thus for any sequence $\{g_n\}_{n\ge 1}$ such that $\|g_n\|_\infty\le 1$ we have
$$\|P_S^{q+1}(g_n)-P_S^{q+1}(g_m)\|_\infty\le \|H_n-H_m\|_\infty,$$
where 
\begin{align*}
&H_n(s_{q+2},\ldots,s_{2q+1})\\
&\qquad :=\int_{\R^{q+1}} g_n(x_{q+2},\cdots,x_{2q+2})\prod_{\ell=q+2}^{2q+2} {\bf 1}_{x_\ell+s_{\ell}+ \sum_{j=1}^{\ell-q-2} a_jx_{\ell-j}>0}  \dd F(x_{\ell}),\end{align*}
with \fa{$s_{\ell}:=\sum_{j=\ell-q-1}^q a_jx_{\ell-j}$ for $\ell\in [q+2,2q+1]$, and $s_{2q+2}:=0$.}
 It thus suffices to show that $H_n$ is Cauchy in sup norm along a subsequence. 
To this end, we consider three sub cases depending on the value of ${\bf s}:=(s_{q+2},\cdots,s_{2q+1})$.

(a) If $s_{\ell}< -L$ for some $\ell\in [q+2,2q+1]$ then
  \begin{align*}
H_n(s_{q+2},\ldots,s_{2q+1})\le \P(\xi_{\ell}+\sum_{j=1}^{\ell-q-2 }a_j \xi_{\ell-j}>L),
\end{align*}
and so given $\varepsilon>0$ there exists $L=L(\varepsilon)<\infty$ such that
\begin{align}\label{eq:11}
\sup_{{\bf s}\in\R^{q}:\min_{q+2\le \ell\le 2q+1}s_\ell<-L}H_n({\bf s})\le \varepsilon.
\end{align}

 (b) If  ${\bf s}\in \R^{q}$ is such that for some $r\in \{1,\ldots,q\}$ the coordinates $s_{\ell_1},\ldots,s_{\ell_r}$ are in $[-L,L]$, and the other coordinates $s_{\ell_{r+1}}, \ldots, s_{\ell_q}$ are larger than $L$, then setting
   $H_{n,\ell_1,\ldots,\ell_r}(s_{\ell_1},\ldots,s_{\ell_r})$ to equal
 $$  \int_{\R^{q+1}} g_n(x_{q+2},\cdots,x_{2q+2})\prod_{k=0}^r {\bf 1}_{x_{\ell_k}+s_{\ell_k}+ \sum_{j=1}^{\ell_k-q-2} a_j x_{\ell_k-j}>0}  \prod_{\ell=q+2}^{2q+2} \dd F(x_{\ell})$$
 we have
$$|H_n(s_0,\ldots,s_q)-H_{n,\ell_0,\ldots,\ell_r}(s_{\ell_0},\ldots,s_{\ell_r})|
\le \P(\cup_{k=0}^r \{\xi_{\ell_k}+\!\sum_{j=1}^{\ell_k-q-2} \!
a_j \xi_{\ell_k-j}<-L\}),$$ and so again by choosing $L$ large enough we can ensure that
\begin{align}\label{eq:12}
\sup_{{\bf s}\in\R^{q}: s_{\ell_k}\in [-L,L], 1\le k\le r, s_{\ell_k}>L, r+1\le k\le q}|H_n({\bf s})-H_{n,\ell_0,\ldots,\ell_r}(s_{\ell_0},\ldots,s_{\ell_r})|\le \varepsilon.
\end{align}

(c) If ${\bf s}\in \R^{q}$ is such that
$s_\ell \in [-L,L]$ for all $\ell\in [q+2,2q+1]$, then we have
\begin{align*}
|H_n({\bf s})-H_n({\bf t})|
&\le \P(\{\min_{q+2\le \ell\le 2q+1}s_\ell+\xi_{\ell}+\sum_{j=1}^{\ell-q-2} a_j\xi_{\ell-j}>0\} \Delta \{\min_{q+2\le \ell \le 2q+1}t_\ell+\xi_{\ell}+\sum_{j=1}^{\ell-q-2} a_j\xi_{\ell-j}>0\}),
\end{align*}
where $\Delta$ is the symmetric difference between two sets.
\fa{Since the right-hand side in the last display
is continuous in the arguments ${\bf s}, {\bf t}$ over the compact set $[-L,L]^{2q}$ (due to the assumption of continuity of $F$), and vanishes on the set ${\bf s}={\bf t}$, it follows that given $\eta>0$ there exists $\delta=\delta(\eta,L)$ such that whenever $\|{\bf s}-{\bf t}\|_\infty<\delta$, we have $|H_n({\bf s})-H_n({\bf t})|<\eta$. In particular, this means that
$H_n(.)$ is uniformly equicontinuous on $[-L,L]^q$.} By the
Arzel\`a-Ascoli theorem, we have that $\{H_n\}_{n\ge 1}$ is compact with respect to sup norm topology on $[-L,L]^{q}$, and so there exists a subsequence which is Cauchy in sup norm. A similar argument applies to each of the functions $H_{n,\ell_0,\ldots,\ell_r}$ for all choices of $r\in \{1,\ldots,q\}$ and $\{\ell_1,\ldots,\ell_r\}$ which are  subsets of $\{q+2,\ldots,2q+1\}$ of size $r$. Thus by going through subsequences, we may assume all the  functions $H_{n,\ell_1,\ldots,\ell_r}$ are Cauchy in sup norm on $[-L,L]^{q}$.

Taking limits along the subsequence from step (c) gives
and using \eqref{eq:11} and \eqref{eq:12} gives
$$\limsup_{m,n\rightarrow\infty}\sup_{{\bf s}\in \R^{q}}|H_n({\bf s})-H_m({\bf s})|\le 2\varepsilon,$$
and so $\{H_n\}_{n\ge 1}$ is Cauchy in sup norm on $\R^{q}$. Thus it follows by an application of Theorem~\ref{thm:exponent} that the operator $P_S$ has largest eigenvalue $\beta_F({\bf a})$. Finally note that $[P_S](g)(x_0,\cdots,x_q)$ is by definition 
independent
of $x_0$, and so w.l.o.g.\ the eigenfunction $\psi$ can be taken to be a function of $q$ variables giving the eigenvalue equation
$$\beta_F({\bf a}) \psi(x_1,\cdots,x_q)\!=\!\int_{y+\sum_{j=1}^q a_j x_{q+1-j}>0}
\!\!\!\!
\!\!\!\!
\!\!\!\!
\!\!\!\!
\psi(x_2,\cdots,x_{q},y)\dd F(y)=[K(\psi)](x_1,\cdots,x_{q}),$$
where $K$ is as defined in the theorem. Thus,
$K$ satisfies the desired eigenvalue equation.

Finally it remains
to check condition (ii) in Theorem~\ref{thm:exponent}. To this end, setting $A$ to be the support of $F$, ${\bf X}$ is a Markov chain on $A^{q+1}$. Since sets of the form $$\{{\bf x}\in A^{q+1}:x_j\in U_j\cap A, 0\le j\le q\}$$ 
with $\{U_j,0\le j\le q\}$  open sets in $\R$ form a base of the topology on $A^{q+1}$ and since
$$\P(\xi_{j}\in U_j\cap A, 0\le j\le q)=\prod_{j=0}^q \P(\xi_j \in U_j\cap A),$$
it suffices to show
that $\P(\xi_0\in U\cap A)>0$ for every open (in $\R$)
set $U$ which  intersects $A$; this follows 
at once
from the assumption of a continuous distribution function. 

We have verified that the conditions of
Theorem \ref{thm:exponent} hold; an application of the latter 
yields the existence of $\beta_F({\bf a})$, and hence completes the proof of Theorem \ref{thm:mainfty}.
\end{proof}

\begin{proof}[Proof of Theorem \ref{thm:ma_conti}]
We recall the notation ${\bf X}(i):=(\xi_{i-q},\ldots,\xi_i)$ and set $S=S({\bf a}):=\{ (x_0,\ldots,x_q) | x_q + \sum_{i=1}^q a_j x_{q-j} > 0\}$.

Let $\{{\bf a}^{(k)}\}_{k\ge 1}$ be a sequence of vectors in $\R^q$ converging to ${\bf a}$. Then for any $1\le m\le n$ setting $M_n:=\lfloor \frac{n}{m+q}\rfloor$ and $$I_j:=\bigl[(j-1)(m+q)+1,(j-1)(m+q)+m\bigr] \; \mbox{\rm for}\;j\geq 1,$$ we have
\begin{align*}
&\P_{{\bf a}^{(k)}}(Z_i>0,0\le i\le n)\le \P_{{\bf a}^{(k)}}(Z_i>0, i\in I_j, 1\le j\le M_n)\\
&=\prod_{j=1}^{M_n}\P_{{\bf a}^{(k)}}(Z_i>0, i\in I_j)
=\P_{{\bf a}^{(k)}}(Z_i>0,1\le i\le m)^{M_n},
\end{align*}
which upon taking $\log$, dividing by $n$, and letting $n\rightarrow\infty$ we obtain that
$$\log \lambda(P_{S({\bf a}^{(k)}}))\le \frac{1}{m+q}\log \P_{{\bf a}^{(k)}}(Z_i>0, 1\le i\le m).$$
Letting $k\rightarrow\infty$ and noting that the distribution of $(Z_1,\cdots,Z_m)$ under ${\bf a}^{(k)}$ converges to the distribution of $(Z_1,\cdots,Z_m)$ under ${\bf a}$ gives
$$\limsup_{k\rightarrow\infty}\log \lambda(P_{S({\bf a}^{(k)}}))\le \frac{1}{m+q}\log \P_{{\bf a}}(Z_i>0, 1\le i\le m),$$
which upon letting $m\rightarrow\infty$ gives
$\limsup_{k\rightarrow\infty}\lambda(P_{S({\bf a}^{(k)})})\le \log\lambda(P_{S({\bf a})}),$
thus giving the upper bound.

We now turn to the lower bound. Fix $M>0$, set $S_M=S_M({\bf a}):=
S({\bf a})\cap [-M,M]^{q+1}$, and invoke Theorem \ref{thm:mainfty} to obtain
\begin{align*}
\P_{{\bf a}^{(k)}}({\bf X}(i)\in S,0\le i\le n)\ge &\P_{{\bf a}^{(k)}}({\bf X}(i)\in S_M,0\le i\le n)=\lambda(P_{S_M({\bf a}^{(k)})})^{n+o(n)},
\end{align*}
where $P_{S_M(\cdot)}$ is viewed as an operator on $\cB([-M,M]^{q+1})$ (and not $\cB(\R^{q+1})$). 
This gives
$ \lambda(P_{S({\bf a}_k)})\ge \lambda(P_{S_M({\bf a}_k)}).$
From this the lower bound will follow via Lemma \ref{lem:conti} if we can show the following:
\begin{align}
\lim_{k\rightarrow\infty}\|P_{S_M({\bf a}^{(k)})}-P_{S_M({\bf a})}\|=0\label{eq:ma_conti_1},\\
\limsup_{M\rightarrow\infty}\lambda(P_{S_M({\bf a})})\ge\lambda(P_S({\bf a}))\label{eq:ma_conti_2}.
\end{align}
To show \eqref{eq:ma_conti_1}, for any $f\in \cB([-M,M]^{q+1})$ such that $\|f\|_\infty\le 1$ we have
$$|(P_{S_M({\bf a}^{(k)})}f-P_{S_M({\bf a})}f)(x_1,\cdots,x_{q+1})| \le \P(A({\bf a}^{(k)})\Delta A({\bf a})|\xi_\ell=x_\ell, 1\le \ell\le q+1),$$
where $\Delta$ denotes symmetric set difference, and $$A({\bf a}):=\{\xi_\ell +\sum_{j=1}^q a_j\xi_{\ell-j}>0, q+2\le \ell \le 2q+2\}.$$ Setting $s_\ell({\bf a},{\bf x}):=\sum_{j=\ell-q-1}^qa_j x_{\ell-j}$ for $q+2\le \ell\le 2q+2$ we have 
\begin{align*}
 & \P(A({\bf a}^{(k)})\Delta A({\bf a})|\xi_\ell=x_\ell, 1\le \ell\le q+1)
 \\
 &=\P(\{\xi_\ell+\sum_{j=1}^{\ell-q-2}a_j^{(k)}{\xi_{\ell-j}}+s_\ell({\bf a}^{(k)},{\bf x})>0\}\Delta \{\xi_\ell+\sum_{j=1}^{\ell-q-2} a_j {\xi_{\ell-j}}+s_\ell({\bf a},{\bf x})>0\}).
 \end{align*}
Since ${\bf a}^{(k)}$ converges to ${\bf a}$  we have $\max_{{\bf x}\in [-M,M]^{q+1}}|s_\ell({\bf a}^{(k)},{\bf x})-s_\ell({\bf a},{\bf x})|=0$, which along with the continuity of distribution functions gives that the RHS above converges to $0$ as $k\rightarrow\infty$, uniformly in $(x_1,\cdots,x_{q+1})\in [-M,M]^{q+1}$, 
and so we have verified \eqref{eq:ma_conti_1}.

Proceeding to verify \eqref{eq:ma_conti_2}, fixing $M,\varepsilon>0$ and invoking Theorem \ref{thm:mainfty}  there exists $N:=N(\varepsilon,M)<\infty$ such that for all $n\ge N$ we have
$$ \P_{{\bf a}}({\bf X}(i)\in S_M,0\le i\le n)\le (\lambda(P_{S_M({\bf a})})+\varepsilon)^n.$$
Thus with $\delta>0$ and $\ell:=\lfloor n\delta\rfloor$ 
we have
\begin{align}
\notag&\fa{\P_{{\bf a}}({\bf X}(i)\in S,0\le i\le n)}\\
\notag\le &\fa{\P_{{\bf a}}(\exists (i_s)\in [0,n]_{\ell}: {\bf X}(i)\in S_M, i\notin \{i_1,\cdots,i_\ell\})}\\
+&\fa{\P_{{\bf a}}(\exists (i_s)\in [0,n]_{n-\ell}: {\bf X}(i_s)\notin S_M, 1\le s\le \ell)}\label{eq:2020_1},
\end{align}
where $[0,n]_\ell$ is the set of all distinct integer tuples of size $\ell$ with entries in $[0,n]$. Fixing an increasing sequence $(i_1,\cdots,i_{\ell})\in [0,n]_\ell$ let $j_s:=i_s-i_{s-1}$ for $1\le s\le\ell $ with $i_0:=0$ we have
\begin{align*}
\P_{{\bf a}}({\bf X}(i)\in S_M,i\notin \{i_1,\cdots,i_s\})
\le &\!\!\prod_{1\le s\le \ell: j_s-q\ge N}\!\!
\P_{{\bf a}}({\bf X}(i)\in S_M,1\le i\le j_s-q)\\
\le &(\lambda(P_{S_M({\bf a})})+\varepsilon)^{\sum_{s: j_s-q> N}(j_s-q)}.
\end{align*}
To estimate the exponent in the RHS above, first note that $\sum_{s=1}^\ell j_s=n+1-\ell\ge n-n\delta$, whereas $\sum_{s:j_s-q\le N}j_s\le (q+N)\ell\le (q+N)n\delta$. Combining these two estimates gives
$$\sum_{s:j_s-q>N}(j_s-q)\ge \sum_{s:j_s-q>N}j_s-q \ell\ge n-n\delta -(2q+N)n\delta,$$
which on using \eqref{eq:2020_1} gives 
\begin{align}
 \notag &\P_{{\bf a}}({\bf X}(i)\in S,0\le i\le n)\\
 \notag &\le \lfloor n\delta \rfloor {n+1\choose \lfloor n\delta \rfloor} (\lambda(P_{S_M({\bf a})})+\varepsilon)^{n-n\delta-(2q+N)n\delta}\\
\label{eq:2020_2}  +&\P_{{\bf a}}(\exists (i_s)\in [0,n]_{n-\ell}: {\bf X}(i_s)\notin S_M, 1\le s\le \ell).
\end{align}
To estimate the second term in the right hand side of \eqref{eq:2020_2}
note that ${\bf X}(i)\notin [-M,M]^{q+1}$ implies $(q+1)\sum_{j=i-q}^i|\xi_i|\ge M$, and so 
\begin{align*}
\P({\bf X}(i_s)\notin S_M, 1\le s\le \ell)\le &\P(
\sum_{i=-q}^n 1\{|\xi_i|>M\}\ge  \frac{M}{q+1})\\
 \le &\P\Big(\mbox{Bin}(n+q+1,\P(|\xi_1|>M))>\frac{n}{q+1}\delta\Big),
\end{align*}
which on taking $\log$, dividing by $n$ and letting $n\rightarrow\infty$ followed by $M\rightarrow\infty$ gives $-\infty$ for every fixed $\delta>0$, and so does not contribute. Thus,  taking $\log$, dividing by $n$ and letting $n\rightarrow\infty$ on both sides of \eqref{eq:2020_2} gives  $$\lim_{n\rightarrow\infty}\frac{1}{n}\log \P_{\bf a}({\bf X}(i)\in S, 1\le i\le n)\le (1-\delta(1+2q+N))\log(\lambda(P_{S_M({\bf a})})+\varepsilon).$$
On letting $M\rightarrow\infty$ followed by $\delta\rightarrow 0$ in the above display gives
$$\log \lambda(P_{S,{\bf a}})\le \liminf_{M\rightarrow\infty}\log \Big(\lambda(P_{S_M({\bf a})})+\varepsilon\Big),$$
which verfies \eqref{eq:ma_conti_2} as $\varepsilon>0$ is arbitrary, and hence completes the proof of the theorem.
\end{proof}

\begin{proof}[Proof of Proposition \ref{thm:mafinite}]
  We start with the ``if'' part. With $a_0:=1$, we have $\sum_{j=0}^qa_j\ne 0$. Assume $\sum_{j=0}^q a_j>0$ (the case $<0$ is analogous), and set $$A:=\sum_{j\in \{0,\ldots,q\}:a_j>0} a_j >  B:=\left|\sum_{j\in\{0,\ldots,q\}:a_j<0}a_j\right|.$$ 
 Since the distribution of $\xi$ is continuous and $\P(\xi_0>0)>0$, there exist $x>0$, $\delta\in \Big(0,\frac{x(A-B)}{B}\Big)$ such that
$\P(\xi_j\in (x,x+\delta))>0.$ 
This gives
 \begin{align*}
 Ax-B(x+\delta)=x(A-B)-B\delta >0,
 \end{align*}
and so
$$\P(\min_{i\in \{0,\ldots,n\}}Z_i>0)\ge \P(\xi_i\in (x,x+\delta),-q\le i\le n).$$
Indeed, to see this note that
\begin{align*}
Z_i=&\xi_i+\sum_{j=1}^q a_j \xi_{i-j}>x \sum_{j\in\{0,\ldots,q\}:a_j>0}|a_j|-(x+\delta)\sum_{j\in \{0,\ldots,q\}:a_j<0}|a_j|>0.
\end{align*}
The desired conclusion then follows on noting that $\P(\xi_i\in (x,x+\delta))>0$.

We continue with the ``only if'' part of the proposition. Define the numbers $b_i:=-\sum_{j=i+1}^q a_j$ and the MA($q-1$)-process $\tilde Z_i:=\sum_{j=1}^{q-1} b_j \xi_{i-j} + \xi_i$ for $i\ge -1$. A short computation using the assumption $\sum_{j=1}^q a_j = -1$ shows that $Z_i = \tilde{Z}_i - \tilde{Z}_{i-1}$ for all $i\ge 0$. Thus,
\begin{eqnarray*}
\P( \min_{0\le i \le n} Z_i > 0 ) &=& \P( \tilde Z_{0} < \tilde Z_1 < \ldots < \tilde Z_n)
\\
&\leq &\P( \tilde Z_{kq} < \tilde Z_{(k+1) q}, k\in\{0,\ldots,\lfloor n/q\rfloor\}).
\end{eqnarray*}
Now note that the random variables $\{\tilde Z_{k q}\}_{k\ge 1}$ are independent (since $\{\tilde Z_i\}_{i\ge 1}$ is $(q-1)$-dependent) and identically distributed (by the stationarity of $\tilde Z$). Thus, \eqref{eqn:krishnapur} shows that the last probability equals $1/(\lfloor n/q\rfloor+1)!$ giving $\beta_F({\bf a})=0$.
\end{proof}

 \subsection{Proof of the results of Section \ref{sec:auto}}

\begin{proof}[Proof of Theorem \ref{thm:ar}]
The proof is based on Theorem \ref{thm:exponent} with $S:=[0,\infty)^p\subset \R^p$, $K=P_S$ and $k=p$, and consists in checking the assumptions there, and in particular the compactness of $K^p$.

(a) If ${\bf a}={\bf 0}$ then the process is i.i.d.\ for which all conclusions are trivial. Thus assume w.l.o.g.\ that ${\bf a}\ne {\bf 0}$, and that $a_p<0$ (otherwise we can reduce the value of $p$).   Note that ${\bf X}(i):=(Z_{i-p+1},\ldots,Z_{i})$ is a Markov chain on $\R^p$. 
Note that
\begin{align*}
&[P_S^p(g)](x_1,\ldots,x_p)=\int\limits_{(0,\infty)^p} g(x_{p+1},\ldots,x_{2p}) \prod_{\ell=p+1}^{2p} \phi(x_\ell-\sum_{j=1}^p a_j x_{\ell-j}) \dd x_{\ell}\\
&\qquad\quad =\int\limits_{(0,\infty)^p}g(x_{p+1},\cdots,x_{2p}) H\Big((x_1,\cdots,x_p),(x_{p+1},\cdots,x_{2p})\Big) \prod_{\ell=p+1}^{2p} \dd x_{\ell},
\end{align*}
where $$ H\Big((x_1,\cdots,x_p),(x_{p+1},\cdots,x_{2p})\Big):= \prod_{\ell=p+1}^{2p} \phi(x_\ell-\sum_{j=1}^p a_j x_{\ell-j})$$ for $(x_1,\cdots,x_p,x_{p+1},\cdots,x_{2p})\in \R^{2p}.$
If $x_\ell>L$ for some $1\le \ell \le p$, then for $(x_1,\cdots,x_p,x_{p+1},\cdots,x_{2p})\in [0,\infty)^{2p}$ we have  $$x_{p+\ell}-\sum_{j=1}^p a_jx_{p+\ell-j}\ge x_{p+\ell}-a_p x_{\ell}\ge x_{p+\ell}-a_pL,$$ and so  given a sequence of functions $\{g_n\}_{n\ge 1}$ such that $\|g_n\|_\infty\le 1$ we have
$$ \sup_{{\bf x}\in S: \|{\bf x}\|_\infty  > L}|[P_S^p(g_n)]({\bf x})|\le 1-F(-a_pL),$$
where we recall that $F$ is the distribution function of the innovation density $\phi$.
Therefore, given $\varepsilon>0$ there exists $L=L(\varepsilon)<\infty$ such that
\begin{align}\label{eq:111}
 \sup_{{\bf x}\in S:\|{\bf x}\|_\infty> L}|[P_S^p(g_n)]({\bf x})|\le\varepsilon.
 \end{align}
On the other hand, for ${\bf x}^{(1)},{\bf x}^{(2)}\in [0,L]^p$ we have
\begin{align}\label{eq:112}
|[P_S^p (g_n)]({\bf x}^{(1)})-[P_S^p (g_n)]({\bf x}^{(2)})|\le \int_{\R^p}\Big| H({\bf x}^{(1)}, {\bf y})-H({\bf x}^{(2)},{\bf y})\Big| \dd {\bf y}.
\end{align}
Now given $\eta>0$ there exists a non negative continuous integrable function $\tilde{\phi}:\R\mapsto \R$ such that
$\int_\R |\phi(x)-\tilde{\phi}(x)|\dd x<\eta$, which in particular implies $\int_\R \tilde{\phi}(x)\dd x\le 1+\eta$. Using this, setting

$\tilde{H}\Big((x_1,\cdots,x_p),(x_{p+1},\cdots,x_{2p})\Big):=\prod_{\ell=p+1}^{2p}\tilde{ \phi}(x_\ell-\sum_{j=1}^p a_j x_{\ell-j})$

and using a telescopic argument we have
\begin{align}
\notag&\int_{\R^p}|H({\bf x},{\bf y})-\tilde{H}({\bf x},{\bf y})|\dd {\bf y}\\
\notag&=\int_{\R^p}\Big|\prod_{\ell=p+1}^{2p} \phi(x_\ell-\sum_{j=1}^pa_jx_{\ell-j})-\prod_{\ell=p+1}^{2p}\tilde{\phi}(x_\ell-\sum_{j=1}^pa_jx_{\ell-j})\Big|\prod_{\ell=p+1}^{2p} \dd x_\ell\\
\notag &\le \sum_{r=1}^p\int_{\R^p}\Biggl[ \prod_{\ell=p+1}^{p+r-1} {\phi}(x_\ell-\sum_{j=1}^pa_jx_{\ell-j}) \prod_{\ell=p+r+1}^{2p}\tilde{\phi}(x_\ell-\sum_{j=1}^pa_jx_{\ell-j})\\
\notag&\qquad\times\Big|\tilde{\phi}(x_{p+r}-\sum_{j=1}^pa_jx_{p+r-j})-\phi(x_{p+r}-\sum_{j=1}^pa_jx_{p+r-j})\Big| \Biggr]\prod_{\ell=p+1}^{2p}\dd x_\ell\\
\notag &\le\sum_{r=1}^p(1+\eta)^{p-r}\eta\int_{\R^{r-1}} \prod_{\ell=p+1}^{p+r-1}\phi(x_\ell-\sum_{j=1}^pa_jx_{\ell-j})\prod_{\ell=p+1}^{p+r-1}\dd x_\ell\\
&=\sum_{r=1}^p(1+\eta)^{p-r}\eta\le \eta p (1+\eta)^p\label{eq:H1}.
\end{align}
Finally, using the continuity of $\tilde{\phi}$ along with Scheffe's Lemma, there exists a $\delta=\delta(\eta, L)$ such that for all ${\bf x}^{(1)},{\bf x}^{(2)}\in [0,L]^p$ with $\|{\bf x}^{(1)}-{\bf x}^{(2)}\|_\infty<\delta$ we have 
\begin{align}\label{eq:H2}
\int_{\R^p}\Big|\tilde{H}({\bf x}^{(1)},{\bf y})-\tilde{H}({\bf x}^{(2)},{\bf y})\Big| \dd{\bf y}<\eta.
\end{align}
Combining \eqref{eq:112}, \eqref{eq:H1}, \eqref{eq:H2}, we conclude that for any ${\bf x}^{(1)},{\bf x}^{(2)}\in [0,L]^p$ such that $\|{\bf x}^{(1)}-{\bf x}^{(2)}\|_\infty<\delta$, we have
$$\Big|[P_S^p(g_n)]({\bf x}^{(1)})-[P_S^p(g_n)]({\bf x}^{(2)})\Big|\le \eta+2\eta p(1+\eta)^p.$$
%
%
Since $\eta>0$ is arbitrary, it follows that the sequence $\{P_S^p(g_n)\}_{n\ge 1}$ is uniformly equicontinuous on $[0,L]^p$.Thus by the
 Arzel\`a-Ascoli theorem, there exists a subsequence along which $\{P_S^p(g_n)\}_{n\ge 1}$ is Cauchy in sup norm  on $[0,L]^p$. Taking limits along this subsequence and using \eqref{eq:111} gives
 $$\limsup_{m,n\rightarrow\infty}\|P_S^p g_n-P_S^p g_m\|_\infty\le 2\varepsilon,$$
 and so we have proved the existence of a convergent subsequence in sup norm on $[0,\infty)$, and thus the compactness of $K^p$. An application of
   Theorem \ref{thm:exponent} then yields part (a).

%
%

   (b) The fact that $\theta_F({\bf a})>0$ follows from Proposition \ref{ppn:finite} along with the assumption that $\mu$ has full support, and $\P_F(\xi_1>0)>0$. For the other inequality, for any non-negative function $g\in \cB([0,\infty)^p)$ such that $\|g\|_\infty\le 1$ we have
\begin{align*}
&P_S(g)(x_1,\cdots,x_p)\\
=&\int_{y+\sum_{j=1}^pa_jx_{p+1-j}>0}g(x_2,\cdots,x_p,y+\sum_{j=1}^pa_jx_{p+1-j}) \phi(y)dy\\
\le&  \P(\xi_1>0)
\end{align*}
and so $\theta_F({\bf a})=\lambda(P_S)\le \|P_S\|_\infty\le  \P(\xi_1>0)<1$.

(c) By assumption we have $\lim_{k\rightarrow\infty}a_p^{(k)}=a_p<0$, and so there exists $\delta>0$ such that $a^{(k)}_p\le -\delta$ for all $k\ge 1$. Along with \eqref{eq:111}, this gives
\begin{align*}
|P_{{{\bf a}^{(k)}},S}^pf({\bf x})-P_{{\bf a},S}^pf({\bf x})|\le2(1-F(\delta L))+2\!\sup_{{\bf x}\in [0,L]^p}\!\|P^p_{{{\bf a}^{(k)}}}({\bf x},.)-P^p_{{\bf a}}({\bf x},.)\|_{TV},
\end{align*}
which on taking a sup over $f$ such that $\|f\|\le1$ and  letting $k\rightarrow\infty$ gives
$\limsup_{k\rightarrow\infty}\|P_{{{\bf a}^{(k)}},S}^p-P_{{\bf a},S}^p\|_\infty\le 2(1-F(\delta L))$. Upon letting $L\rightarrow\infty$ we have $\lim_{k\rightarrow\infty}\|P_{{{\bf a}^{(k)}},S}^p-P_{{\bf a},S}^p\|=0$, which,
 using Lemma \ref{lem:conti},
gives the desired conclusion.

\end{proof}

The proof of
Proposition \ref{ppn:ar_not_strict} uses the next 
lemma, which is an adaptation of 
\cite[Theorem 1.6]{DM} and \cite[Lemma 3.1]{DM2}
for discrete time Gaussian processes. Because of the discreteness of 
the involved processes, we are able to verify continuity of the persistence exponent of any stationary Gaussian process (with non negative correlations) in its levels (c.f. \eqref{eq:conti_level_gauss}), thereby removing one of the conditions of \cite[Theorem 1.6]{DM}.
\begin{lem}\label{lem:ref}
For all $k\ge 1$ let
 $\{Z_k(i)\}_{i\ge 0}$ be
a discrete time centered Gaussian sequence
with non negative covariance function $A_k(\cdot,\cdot)$,
such that 
\begin{align}\label{eq:conv}
\lim_{k\rightarrow\infty}\sup_{i\ge 0}|A_k(i,i+\tau)-A(\tau)|=0,\qquad {\text{for all $\tau\geq 0$,}}
\end{align}
for some function $A(\cdot)$. 
 Suppose further that 
\begin{align}
\label{eq:decay}\limsup_{k,\tau\rightarrow\infty}\sup_{i\ge 0}\frac{\log A_k(i,i+\tau)}{\log \tau}<-1.
\end{align}
Then for every $r\in \R$ we have
$$\lim_{k,n\rightarrow\infty}\frac{1}{n}\log \P(\min_{0\le i\le n}Z_k(i)>r)=\lim_{n\rightarrow\infty}\frac{1}{n}\log \P(\min_{0\le i\le n}Z(i)>r),$$
where $\{Z(i)\}_{i\ge 0}$ is a centered stationary Gaussian sequence with covariance $A(i-j)$. 
\end{lem}
\begin{proof}
To begin note that the proof of the continuous case \cite[Theorem 1.6]{DM} goes through verbatim  in the discrete case under \eqref{eq:decay} and the extra assumptions of \cite[Theorem 1.6]{DM}, namely
that  for every $r\in \R$ we have
\begin{align}\label{eq:conti_level_gauss}
\lim_{\varepsilon\downarrow 0}\lim_{n\rightarrow0}\frac{1}{n}\log \P(\min_{0\le i\le n}Z(i)>r-\epsilon)=\lim_{n\rightarrow0}\frac{1}{n}\log \P(\min_{0\le i\le n}Z(i)>r);
\end{align}
and that  for every $z\in\R$ and positive integer $M$, we have
\begin{align}
\notag\P(\sup_{0\le \tau\le M}Z(\tau)<z)\le &\liminf_{k\rightarrow\infty}\inf_{i\ge 0}\P(\sup_{0\le \tau\le M}Z_k(i+\tau )<z)\\
\label{eq:weak}\le &\limsup_{k\rightarrow\infty}\sup_{i\ge 0}\P(\sup_{0\le \tau\le M}Z_k(i+\tau )<z)\\
\notag\le &\P(\sup_{0\le \tau\le M}Z(\tau)\le z).
\end{align}
\cite[Theorem 1.6]{DM} considers the case $r=0$, but a similar argument applies for any $r\in \R$. It thus remains to verify these two extra conditions, of which \eqref{eq:weak} follows from \eqref{eq:conv}.
 Proceeding to verify \eqref{eq:conti_level_gauss}, 
fixing $\varepsilon,\delta>0$ and setting $\ell:=\lceil n\delta\rceil$ and intersecting with the set $\{|i\in [0,n]:Z(i)\in (r-\varepsilon,r]|>n\delta\}$ and its complement, we have
\begin{align}
\label{eq:conti_gauss_1}
&\fa{\P(\min_{0\le i\le n}Z(i)>r-\varepsilon)}  \\ 
& \fa{\le \P(\exists (i_s)\in [0,n]_{n-\ell}:\min_{1\le s\le\ell}Z(i_s)>r)}\\
\notag&\fa{\qquad+\P(\exists (i_s)\in [0,n]_{\ell}:Z(i_s)\in (r-\varepsilon,r], 1\le s\le \ell)}\\
&\fa{\notag \le\sum_{(i_s)\in [0,n]_{n-\ell}}\P(Z(i_s)>r)
+\sum_{(i_s)\in [0,n]_{\ell}}\P(Z(i_s)>(r-\varepsilon,r],1\le s\le \ell),}
\end{align}

%
where $[0,n]_\ell$ is the set of all integer tuples in $[0,n]^{\ell+1}$ with all entries distinct. 
For estimating the first term in the RHS of \eqref{eq:conti_gauss_1},  on the set $\{Z(i_s)\in (r-\varepsilon,r],0\le s\le \ell\}$  there must be at least $\ell'$ indices $\{j_1,\cdots,j_{\ell'}\}$ with $j_s-j_{s-1}\ge N$, where $\ell'\ge \frac{\lfloor n\delta\rfloor}{N}$, for any $N\ge 1$. Thus, if $B$ denotes the covariance matrix of $\{Z(j_1),\cdots,Z(j_\ell')\}$, we have
$$\max_{i=1}^{\ell'}\sum_{j:j\ne i}B(i,j)\le 2\sum_{i=N}^\infty g(i),$$
where $g$ is a summable function satisfying $$\sup_{i\ge 0,k\ge 1}A_k(i,i+\tau)\le g(\tau)$$ for all $\tau\ge 0$, the existence of which is guaranteed by \eqref{eq:decay}. 
\textcolor{black}{By choosing $N$ large enough we can ensure that 
  $$\max_{1\le i\le \ell'}\sum_{j:j\ne i}B(i,j)\le \frac{B(i,i)}{2}=\frac{A(0)}{2}.$$
  Consequently, by the Gershgorin Circle Theorem all eigenvalues of $B$ lie within $[A(0)/2,3A(0)/2]$. This gives
\begin{align*}
&\P(Z(j_s)\in (r-\varepsilon,r], 1\le s\le \ell')=\frac{\int_{(r-\varepsilon,r]^{\ell'}} e^{-z'B^{-1}z/2}dz}{\int_{\R^{\ell'} }e^{-z'B^{-1}z/2}dz}\\
&\qquad  \le  \frac{\int_{(r-\varepsilon,r]^{\ell'}} e^{-z'z/3A(0)}dz}{\int_{\R^{\ell'} }e^{-z'z/A(0)}dz}
\le 3^{\ell'/2}\P(\cN(0,3A(0)/2)\in (r-\varepsilon,r])^{\ell'}\\
&\qquad \le 3^{\frac{n+1}{2}}\P(\cN(0,3A(0)/2)\in (r-\varepsilon,r])^{\frac{n\delta}{N}}.
\end{align*}
Plugging this in the first term of the RHS of \eqref{eq:conti_gauss_1} we get the bound
$$(n+1)(2\sqrt{3})^{n+1}\P(\cN(0,3A(0)/2)\in (r-\varepsilon,r])^{\frac{n\delta}{N}},$$
which on taking $\log$, dividing by $n$ and letting $n\rightarrow\infty$ followed by $\varepsilon\rightarrow 0$ gives $-\infty$, for every $\delta>0$ fixed. Thus this term does not contribute to the limit.}
For estimating the second term in the RHS of \eqref{eq:conti_gauss_1}, 
using the non negativity of the correlation function along with Slepian's inequality, we get
$$\P(\min_{0\le s\le n-\ell}Z(i_s)>0)\le \frac{\P(\min_{0\le i\le n}Z(i)>0)}{\P(Z(0)>0)^{n\delta}}.$$ This gives the bound $$(n+1){n+1\choose \lceil n\delta\rceil}\frac{\P(\min_{0\le i\le n}Z(i)>0)}{\P(Z(0)>0)^{n\delta}}$$
for the second term in the RHS of
\eqref{eq:conti_gauss_1}. Taking $\log$, dividing by $n$ and letting $n\rightarrow\infty$ followed by $\delta\rightarrow0$, we conclude that
$$\lim_{\varepsilon\downarrow 0}\lim_{n\rightarrow\infty}\frac{1}{n}\log\P(\min_{0\le i\le n}Z(i)>r-\varepsilon) \le \lim_{n\rightarrow\infty}\frac{1}{n}\log \P(\min_{0\le i\le n}Z(i)>r),$$
verifying \eqref{eq:conti_level_gauss}, and hence completing the proof of the lemma.
\end{proof}
\begin{proof}[Proof of Proposition \ref{ppn:ar_not_strict}]
(a)
In this case $\{Z_i\}_{i\ge 0}$ is a stationary Gaussian sequence with non-negative summable correlations, and the conclusions follow from Slepian's Lemma along with sub-additivity. 
\\
(b)
 Note that \eqref{eq:condition_1} implies the existence of
 a sequence of positive reals $\{x_k\}$ diverging to $+\infty$ such that
 $\P_\mu(Z_0>x_k)\ge x_k^{-\varepsilon_k}$,  where $\{\varepsilon_k\}$ is a positive sequence converging to $0$. W.l.o.g., by replacing $\varepsilon_k$ by $\max(\varepsilon_k,\frac{1}{\log x_k})$  if necessary, we can also assume that $\sqrt{\varepsilon}_k\log x_k$ diverges to $+\infty$. Setting $m_k:=\lfloor \sqrt{\varepsilon_k}\log x_k\rfloor $ and fixing $a_1\in (0,1)$, for any $M<\infty$, for all $k$ large enough we have
\begin{align*}
  \P(\min_{0\le i\le m_k}Z_i\geq 0)\ge \P( Z_0>x_k ,|\xi_i|\le M,1\le i\le m_k).
\end{align*}
Indeed, this is because on this set we have
\begin{align*}
Z_i=a_1^{ {i}} Z_0+\sum_{j=0}^{i-1} a_1^{j}\xi_{i-j}> a_1^{m_k} 
x_k-\frac{M}{1-a_1}\os{$k\to\infty$}{\longrightarrow} \infty.
\end{align*}
Therefore,  {for $k$ large enough (depending on $a_1$ and $M$)}
   \begin{align*}
  \P(\min_{0\le i\le m_k}Z_i\geq 0)\ge& \P(Z_0>x_k)\P(|\xi_1|\le M)^{m_k}\ge x_k^{-\varepsilon_k} \P(|\xi_1|\le M)^{m_k}, 
\end{align*}
implying that
$$\liminf_{k\rightarrow\infty}\frac{1}{m_k}\log \P(\min_{0\le i\le m_k}Z_i\geq 0)\ge \log \P(|\xi_1|\le M).$$
 Upon letting $M\rightarrow\infty$, \eqref{eq:no_limit_1} follows. 
 
 Proceeding to verifying
  \eqref{eq:no_limit_2}, use \eqref{eq:condition_2} to get the existence of a sequence of positive reals $\{y_k\}_{k\ge 1}$ diverging to $+\infty$, such that
 $\P_\mu(Z_0>y_k)\le y_k^{-N_k}$,  where $\{N_k\}$ is a sequence of positive reals diverging to $+\infty$. Set $n_k:=\lceil \sqrt{N_k}\log y_k\rceil$, and for  any $\delta>0$ set $n_k':=\lceil \frac{\log y_k-\log \delta}{\log (1/a_1)}\rceil$ to note that
\begin{align}\label{eq:no_limit_3}
\notag\P(\min_{0\le i\le n_k}Z_i\geq 0)\le &\P(Z_0>y_k)+\P(0\le Z_0\le y_k, \min_{1\le i\le n_k}Z_i  {\ge} 0)\\
 \le &y_k^{-N_k}+\P(\min_{0\le i\le n_k-n_k'}Y_{i+n_k'}  {\ge} -\delta),
 \end{align}
 where  $Y_i:=\sum_{\ell=1}^i a_1^{i-\ell}\xi_\ell$ for $1\le i\le n$, and we use the fact that $n_k'<n_k$ for all $k$ large enough. Since the first term in the RHS of \eqref{eq:no_limit_3} on taking $\log$, dividing by $n_k$ and letting $k\rightarrow\infty$ gives $-\infty$, it suffices to consider the second term. \textcolor{black}{To this effect, note that the sequence $\{Y_{i+n_k'}\}_{i\ge 0}$ is a Gaussian sequence with non negative covariance $${A}_k(i,i+\tau):=a_1^{\tau}(1+a_1^2+\cdots+a_1^{2(i+n_k'-1)}),$$ which 
 satisfies \eqref{eq:conv} of Lemma \ref{lem:ref} with $A(\tau):=a_1^\tau(1-a_1^2)^{-1}$.} Also, \eqref{eq:decay} is immediate as well, and so by an application of Lemma \ref{lem:ref}, for any $r\in \R$ we have
%
%
 the second term in the RHS of \eqref{eq:no_limit_3} satisfies
 \begin{align}\label{eq:2020_3}
 \lim_{k\rightarrow\infty}\frac{1}{n_k}\log \P(\min_{n_k'\le i\le n_k}Y_i {\ge}r)=\lim_{n\rightarrow\infty}\frac{1}{n}\log \P_{\NN}(\min_{0\le i\le n}Z(i) {\ge}r).
 \end{align}
 Using this with $r=-\delta$ along with \eqref{eq:no_limit_3} gives the upper bound in \eqref{eq:no_limit_2}, upon 
 letting $\delta\rightarrow 0$ and noting that \eqref{eq:conti_level_gauss} holds for any centered stationary Gaussian process with non negative summable correlations, as shown in the proof of Lemma \ref{lem:ref}. 
 
 To get the lower bound of \eqref{eq:no_limit_2}, note that
  \begin{align}
 \notag \P(\min_{0\le i\le n_k}Z_i\geq 0)\ge & \P(Z_0\ge 0, \min_{1\le i\le n_k'}Y_i>0, \min_{0\le i\le n_k-n_k'}Y_{i+n_k'}>0)\\
\notag  \ge &\P(Z_0\ge 0, \min_{1\le i\le n_k'}\xi_i>0, \min_{0\le i\le n_k-n_k'}Y_{i+n_k'}>0)\\
 \label{eq:lb_2020} \ge &\P(Z_0\ge 0) \P(\xi_1>0)^{n_k'} \P(\min_{0\le i\le n_k-n_k'}Y_{i+n_k'}>0),
  \end{align}
  where the last inequality uses positive association. The lower bound of \eqref{eq:no_limit_2} follows from \eqref{eq:lb_2020}, on noting that the the first two terms in \eqref{eq:lb_2020} do not contribute to the limit, and the third term can be analyzed using \eqref{eq:2020_3} with $r=0$.\\
 (c)
 The proof of part (c) follows by a similar argument as that of part (b). For the upper bound, fixing $x_0\in \R$ and $\delta>0$, setting $N:=\max\Big\{1,\big\lceil \frac{\log |x_0|-\log \delta}{\log \frac{1}{a_1}}\big\rceil\Big\}$ we have
 \begin{align*}
 \P(\min_{1\le i\le n}Z_i\ge 0|Z_0=x_0)\le \P(\min_{N\le i\le n}Y_i\ge -\delta)
 \end{align*}
 where $Y_i=\sum_{\ell=1}^i a_1^{i-\ell}\xi_\ell$.
 An argument similar to the proof of part (b), 
 using Lemma \ref{lem:ref} with $r=-\delta$, then gives 
 $$\lim_{\delta\downarrow 0}\lim_{n\rightarrow\infty}\frac{1}{n}\log \P(\min_{N\le i\le n}Y_i\ge -\delta)=\lim_{n\rightarrow\infty}\frac{1}{n}\log \P_{\NN}(\min_{0\le i\le n}Z(i) {\ge}0).$$
 For the lower bound, using Slepian's Lemma gives
 \begin{align*}
 \P(\min_{1\le i\le n}Z_i\ge 0|Z_0=x_0)\ge & \P(\min_{1\le i\le N-1}Z_i\ge 0|Z_0=x_0)\P(\min_{N\le i\le n}Z_i\ge 0|Z_0=x_0).
 \end{align*}
 Invoking Lemma 
 \ref{lem:ref} with $r=0$ controls the second factor in the RHS of the last display,
 whereas the first factor remains bounded away from $0$ and therefore 
 does not contribute to the limit. 
\end{proof}

\begin{proof}[Proof of Theorem \ref{thm:ar+}]
(a)
By assumption there exists $\delta_1>0$ such that $\E e^{\delta_1 \sum_{j=0}^{p-1}Z_j}1_{\{\min_{0\le j\le p-1}Z_j>0\}}<\infty$.  Invoking \eqref{eq:+} gives the existence of $\delta_2>0$ and $C_2<\infty$ such that $\phi(t)\le C_2 e^{-\delta_2|t|}$. Fix $\delta<\min(\delta_1,\delta_2/p)$,  and set $h({\bf x}):=e^{\delta\sum_{j=0}^{p-1}x_j}$.
For any $k\geq 1$, define 
  $$F(x_1,\cdots,x_{(k+1)p}):=\frac{h(x_{kp+1},\cdots,x_{(k+1)p})}{h(x_1,\cdots,x_p)}\prod_{\ell=p+1}^{(k+1)p}\phi(x_\ell-\sum_{j=1}^pa_jx_{\ell-j}).$$
We have the following lemma, whose proof is deferred.
%
\begin{lem}\label{lem:k+}
  There exist $k\ge 1$  and $C,\gamma>0$ such that for any vector $(x_1,\cdots,x_{(k+1)p})\in  [0,\infty)^{(k+1)p}$ we have
\begin{equation}\label{eq:k+}
F(x_1,\cdots,x_{(k+1)p})
\le C e^{-\gamma\sum_{j=1}^{(k+1)p}x_j}. 
\end{equation}
Further, the constants $C$ and $\gamma$ can be chosen uniformly over the parameter space $\sum_{j=1}^pa_j\le 1-\eta$ for any $\eta>0$.
\end{lem}
Continuing with the proof of Theorem \ref{thm:ar+},
define the operator $Q_{\delta,k}$ on $\cC_b([0,\infty)^p)$ by 
$$Q_{\delta,k} g(x_1,\ldots,x_p)=\int_{[0,\infty)^{kp}}g(x_{kp+1,\cdots,(k+1)p})F(x_1,\cdots,x_{(k+1)p})  \prod_{\ell=p+1}^{(k+1)p} \dd x_\ell.$$
Lemma \ref{lem:k+} ensures that $Q_{\delta,k}{\bf 1}$ is a bounded function, and hence $Q_{\delta,k}$ is well defined on $\cB([0,\infty)^p)$. Note that $gh$ is bounded, and $Q_{\delta,k}(g)=\frac{1}{h} P_S^{kp}(gh)$ by definition, which on using induction gives $Q_{\delta,k}^i(g)=\frac{1}{h}P_S^{kpi}(gh)$ for $i\ge 1$. Thus we have
$${p_n=\int P_S^n{\bf 1} \dd \mu \le \int h Q_{\delta,k}^{\lfloor \frac{n}{kp}\rfloor}(\frac{1}{h}) \dd \mu} \le \|Q_{\delta,k}^{\lfloor \frac{n}{kp}\rfloor}\| \mu(h1_{[0,\infty)^p}),$$
which, because of  the assumption $\mu(h1_{\{[0,\infty)^p\}})<\infty$, results in
\begin{align}\label{eq:ub+}
\limsup_{n\rightarrow\infty}\frac{1}{n}\log p_n\le \frac{1}{kp}\log \lambda(Q_{\delta,k}).
\end{align}
For the lower bound on the persistence probability, fixing $M>0$ we have
\begin{align}\label{eq:lb+}P_S^n{\bf 1}\ge  Q_{\delta,k}^{\lceil \frac{n}{kp}\rceil}(\frac{1}{h})\ge e^{-p\delta M} Q_{\delta,k,M}^{\lceil \frac{n}{kp}\rceil}({\bf 1}_{[0,M]^p}),
\end{align}
where $Q_{\delta,k,M}$ is the operator from $\cB([0,\infty)^p)$ to itself
given by
\begin{align*}
  &Q_{\delta,k,M} g(x_1,\ldots,x_p)=\\
&1_{\{\max_{1\le \ell \le p}x_\ell\le M\}}
\int_{[0,M]^{kp}}g(x_{kp+1,\cdots,(k+1)p})F(x_1,\cdots,x_{(k+1)p})  \prod_{\ell=p+1}^{(k+1)p} \dd x_\ell.
\end{align*}
For any $f\in \cB([0,\infty)^p)$ with $\|f\|_\infty\le 1$, and $\max_{1\le \ell\le p}x_\ell >M$ we  have that
\begin{align*}
&|Q_{\delta,k,M}f(x_1,\cdots,x_p)-Q_{\delta,k}f(x_1,\cdots,x_p)|
\le Q_{\delta,k} {\bf 1}(x_1,\cdots,x_p) \\ 
= &\int_{[0,\infty)^{kp}} F(x_1,\cdots,x_{(k+1)p}) \prod_{\ell=p+1}^{(k+1)p}\dd x_\ell
\le Ce^{-\gamma M}\gamma^{-kp},
\end{align*}
where the last inequality is due to  \eqref{eq:k+}.
Similarly,  if $\max_{1\le \ell\le p}x_\ell\le M$, then 
\begin{align*}
&|Q_{\delta,k,M}f(x_1,\cdots,x_p)-Q_{\delta,k}f(x_1,\cdots,x_p)|\\
\le & \int_{\max_{p+1\le \ell\le (k+1)p}x_\ell>M } F(x_1,\cdots,x_{(k+1)p})\prod_{\ell=p+1}^{(k+1)p}\dd x_\ell
\le C kp \gamma^{-kp} e^{-\gamma M}.
\end{align*}
Combining these two estimates gives $\lim_{M\rightarrow\infty}\|Q_{\delta,k,M}-Q_{\delta,k}\|=0,$ and consequently Lemma \ref{lem:conti} gives
\begin{align}\label{eq:lim+}
\lim_{M\rightarrow\infty}\lambda(Q_{\delta,k,M})=\lambda(Q_{\delta,k}).
\end{align}
Denoting by $\tilde{Q}_{\delta,k,M}$ the restriction of 
the operator $Q_{\delta,k,M}$ to $\cB([0,M]^p)$, for any function 
$f\in \cC_b([0,M]^p)$ such that $\|f\|_\infty\le 1$ we have 
\begin{align*}
&
\!\!\!\!
\!\!\!\!
|\tilde{Q}_{\delta,k,M}f(x_1,\cdots,x_p)-\tilde{Q}_{\delta,k,M}f(y_1,\cdots,y_p)|M^{-kp}\\
\le &\sup_{z_\ell \in [0,M],p+1\le \ell \le (k+1)p}
|F(x_1,\cdots,x_p,z_{p+1},\cdots,z_{(k+1)p})-\\&
\quad\quad\quad\quad
\quad\quad\quad\quad
\quad\quad\quad\quad
\quad\quad\quad\quad
F(y_1,\cdots,y_p,z_{p+1},\cdots,z_{(k+1)p}|,
\end{align*}
and so the functions $\{\tilde{Q}_{\delta,k,M}f: f\in \cC_b([0,M]^p), \|f\|_\infty\le 1\}$ are  
uniformly equicontinuous on $[0,M]^p$.
Consequently, Arzel\`a-Ascoli theorem gives that $\tilde{Q}_{\delta,k,M}$ is 
compact as an operator on $\cC_b([0,M]^p)$. An application of 
  Theorem \ref{thm:exponent} then gives
  $${\int Q_{\delta,k,M}^{\lceil \frac{n}{kp}\rceil}({\bf 1}_{[0,M]^p})\dd \mu } =\lambda(\tilde{Q}_{\delta,k,M})^{\frac{n}{kp}+o(n)},$$
which along with \eqref{eq:lb+} gives
\begin{align}\label{eq:lb+1}
\liminf_{n\rightarrow\infty}\frac{1}{n}\log p_n\ge \frac{1}{kp}\log 
\lambda(\tilde{Q}_{k,\delta,M}).
\end{align}
\fa{Also note that for any $i\ge 1 $ we have 
$Q^i_{\delta,k,M}({\bf 1})={Q}^i_{\delta,k,M}({\bf 1}_{[0,M]^p})$, which immediately gives $\lambda(\tilde{Q}_{\delta,k,M})$ $=\lambda(Q_{\delta,k,M})$.}
The lower bound follows from this, on invoking \eqref{eq:lim+} and \eqref{eq:lb+1}. Thus we have verified that the log persistence exponent exists and equals $\lambda(Q_{\delta,k})^{1/kp}$, which a priori can depend on $\delta>0$. However the above argument works for any 
$\delta<\min(\delta_1,\delta_2/p)$, and 
so the persistence exponent does not depend on $\delta_1$, and in particular does not depend on the initial distribution as long as the latter has finite exponential moment.

For relating $\lambda(Q_{\delta,k})$ to
the eigenvalue equation, we  will invoke Theorem \ref{thm:exponent}, 
for which we need to show that $Q_{\delta,k}$ is 
compact on $\cC_b([0,\infty)^p)$. Since $\|Q_{\delta,k,M}-Q_{\delta,k}\|$ converges to $0$ as $M\rightarrow\infty$, it suffices to show that $Q_{\delta,k,M}$ is compact on $\cC_b([0,\infty)^p)$. Also as shown for the derivation of \eqref{eq:lim+} we have
$$
\sup_{\max_{1\le \ell \le p}x_\ell >M}|Q_{\delta,k}{\bf 1}(x_1,\cdots,x_p)|\le C e^{-\gamma M} \gamma^{-kp},
$$
which converges to $0$ as $M\rightarrow\infty$. This along with the compactness of $\tilde{Q}_{\delta,k,M}$ on $\cC_b([0,M]^p)$ implies that $Q_{\delta,k}$ is compact on $\cC_b([0,\infty)^p)$. If $\lambda(Q_{\delta,k})>0$,  Theorem~\ref{thm:exponent} 
  implies that
  there exists $\tilde{\psi}\in \cC_b([0,\infty)^p)$ such that  $\|\tilde{\psi}\|_\infty=1$, which is non-negative and satisfies
  \begin{align}\label{eq:eigen+}\int_{[0,\infty)^{kp}}\tilde{\psi}(x_{kp+1},\cdots,x_{(k+1)p})F(x_1,\cdots,x_{(k+1)p}) \prod_{\ell=p+1}^{(k+1)p}\dd x_\ell=\lambda^{kp} \tilde{\psi}(x_1,\cdots,x_p),
\end{align}
where $\lambda:=\lambda(Q_{\delta,k})^{1/kp}$.
Using Lemma  \ref{lem:k+} along with \eqref{eq:eigen+} gives
$$\tilde{\psi}(x_1,\cdots,x_p)\le \frac{C}{\lambda^{kp} \gamma^{kp}} e^{-\gamma \sum_{i=1}^p x_i}.$$
Plugging this bound back in \eqref{eq:eigen+}, another application of Lemma \ref{lem:k+}  gives
\begin{align*}
\lambda^{kp}\tilde{ \psi}(x_1,\cdots,x_p)=&\int_{(0,\infty)^{kp}}
\!\!\!
\tilde{\psi}(x_{kp+1},\cdots,x_{(k+1)p}) F(x_1,\cdots,x_{(k+1)p}) \prod_{\ell=p+1}^{(k+1)p}\dd x_\ell\\
\le &\frac{C}{\lambda^{kp} \gamma^{kp}} \int_{(0,\infty)^{kp}} e^{-\gamma \sum_{i=1}^p x_{kp+i}}F(x_1,\cdots,x_{(k+1)p}) \prod_{\ell=p+1}^{(k+1)p}\dd x_\ell\\
\le &\frac{C^2}{\lambda^{kp} \gamma^{2kp}} e^{-2\gamma \sum_{i=1}^p x_i},
\end{align*}
which gives $$\tilde{\psi}(x_1,\cdots,x_p)\le \frac{C^2}{\lambda^{2kp} \gamma^{2kp}} e^{-2\gamma \sum_{i=1}^p x_i}.$$
This, via an inductive argument gives that $\tilde{\psi}(x_1,\cdots,x_p)$ has super exponential decay, and so the function $h\tilde{\psi}\in \cC_b([0,\infty)^p)$. Thus $P_S^{kp}(h\tilde{\psi})$ is well defined, and \eqref{eq:eigen+} shows that $h\tilde{\psi}$ is an eigenfunction of $P_S^{kp}$ with eigenvalue $\lambda^{kp}$. It then follows by a telescopic argument similar to Theorem \ref{thm:exponent} that $P_S$ has an eigenvalue $\lambda$, and the corresponding eigenfunction $\psi\in \cC_b([0,\infty)^p)$ is non-negative.

  (b) As before, $\lambda>0$ follows from Proposition \ref{ppn:finite}. To show that $\lambda<1$, assume by way of contradiction that $\lambda=1$. Invoking part (a) there exists a non zero non-negative function $\psi$ on $[0,\infty)^p$ such that $\|\psi\|_\infty=1$ and $\psi=P_S\psi$. This implies $\psi=P_S^p\psi$.
By the proof of part (a), it follows that $\psi$ has super exponential tails, and so $\psi$ vanishes at $\infty$. Letting $$A:=\{(x_1,\cdots,x_p)\in [0,\infty)^p:\psi(x_1,\cdots,x_p)=\|\psi\|_\infty=1\},$$ we thus have that $A$ is compact, and so $\max(x_1,\cdots,x_p)$ has a  minimum on $A$. 
If the minimum equals $0$, then $(x_1,\cdots,x_p)={\bf 0}$ is a global maximum, and so plugging in $(x_1,\cdots,x_p)={\bf 0}$ we have
$$1=\int_{(0,\infty)} \psi(0,\cdots,0,y) \phi(y)\dd y\le \P(\xi_1>0),$$
which is a contradiction. Thus w.l.o.g.\ we may assume that the minimum of $\max(x_1,\cdots,x_p)$ over $A$ is $m>0$, and fix $(x_1,\cdots,x_p)\in A$ such that  $\max_{1\le i\le p}x_i=m$. Since
$$1=\psi(x_1,\cdots,x_p)= \E \Big(\psi(X_{p+1},\cdots,X_{2p})|X_1=x_1,\cdots,X_p=x_p\Big),$$
we must have 
\begin{align}
\P(X_{p+1}>0,\cdots,X_{2p}>0|X_1=x_1,\cdots,X_p=x_p)&=1\label{eq:prob1},\\
\P(\psi(X_{p+1},\cdots,X_{2p})=1|X_1=x_1,\cdots,X_p=x_p)&=1\label{eq:psi1}.
\end{align}
Since $\P(\xi_1<0)>0$, there exists $c<0$ such that for every $\varepsilon>0$ we have $\P(\xi_1\in [c-\varepsilon,c+\varepsilon])>0$, which along with \eqref{eq:prob1} gives
\begin{align}\label{eq:prob2}
  \P(\cap_{\ell=p+1}^{2p} \{X_{\ell}>0,
  |\xi_{\ell}-c|\leq \varepsilon\}|X_1=x_1,\cdots,X_p=x_p)>0.
\end{align}
Define the tuple $(x_{p+1},\cdots,x_{2p})$ by inductively setting $x_\ell:=c+\sum_{j=1}^p a_j x_{\ell-j}$ for $p+1\le \ell \le 2p$, and note that on the 
set $\{\cap_{\ell={p+1}}^{2p}|\xi_{\ell}-c|\leq \varepsilon\}$ we have $|X_{\ell+1}-x_{\ell+1}|\le \varepsilon+\max_{p+1\le j\le \ell}|X_\ell-x_\ell|$, $p+1\le \ell \le 2p$, which on using induction gives 
$$
\max_{p+1\le\ell\le 2p}|X_\ell-x_\ell|\le p\varepsilon.
$$
  We now claim that $(x_{p+1}.\cdots,x_{2p})\in [0,\infty)^p$. Indeed, if $x_\ell<0$ for some $p+1\le \ell \le 2p$, then by choosing $\varepsilon$ small we have $x_\ell+p\varepsilon<0$, and so
$X_\ell\le x_\ell+p\varepsilon<0$, which is a contradiction to
\eqref{eq:prob2}.
  
  We finally claim that  $\psi(x_{p+1},\cdots,x_{2p})=1$. By way of contradiction, if $\psi(x_{p+1},\cdots,x_{2p})<1$, there exists $\varepsilon>0$ such that $\psi(y_{p+1},\cdots,y_{2p})<1$ for all $(y_1,\cdots,y_p)$ such that $\max_{p+1\le \ell\le 2p} |y_\ell-x_\ell|\le p\varepsilon$. But then we have 
   $$\P(\psi(X_{p+1},\cdots,X_{2p})<1|X_1=x_1,\cdots,X_p=x_p)
   \ge \P(\{\cap_{\ell=p+1}^{2p}
   |\xi_{\ell}-c|\leq \varepsilon\})>0,$$
which is a contradiction to \eqref{eq:psi1}, verifying  $(x_{p+1},\cdots,x_{2p})\in A$. \fa{Finally, using $c<0$ gives
$x_{p+1}=c+\sum_{j=1}^p a_j x_{p+1-j}<m$. Continuing this by induction it is easy to check that  $\max(x_{p+1},\ldots,x_{2p})<m$, which is a contradiction. Thus we have verified that $\lambda<1$.}

(c)
Let $\{{\bf a}^{(r)}\}_{r=1}^\infty $ be a sequence in $\R^p$ converging to ${\bf a}^{(\infty)}$ such that $\sum_{j=1}^p|a_j^{(\infty)}|<1$,
and define $Q_{\delta,k}^{(r)}$ and $F_r$ accordingly, where $k$ is as in Lemma \ref{lem:k+}.
Then there exists $\eta>0$ such that $\sum_{j=1}^p|a_j^{(r)}|\le 1-\eta$ for all $r$ large enough. By Lemma \ref{lem:k+}  the constants $C,\gamma$ depend only on $\eta$, and hence we can choose $C,\gamma$ which works for all $r\ge 1$, and so we have the bound
\begin{align}\label{eq:r_unif}
\sup_{r\ge 1} F_r(x_1,\cdots,x_{(k+1)p})\le C e^{-\gamma\sum_{i=1}^{(k+1)p}x_i} .
\end{align}
Now fix $f\in \cB([0,\infty)^p)$ such that $\|f\|_\infty\le 1$, and $M>0$. If we have $\max_{1\le \ell \le p}x_\ell>M$, then invoking \eqref{eq:r_unif} gives
\begin{align}\label{eq:2019_1}
|Q_{\delta,k}^{(r)}f(x_1,\cdots,x_p)-Q_{\delta,k}^{(\infty)}f(x_1,\cdots,x_p)|\le 2C e^{-\gamma M}
\gamma^{-kp}
\end{align}
If $\max_{1\le \ell \le p}x_\ell\le M$, then splitting the integral depending on whether the integration is over $[0,M]^{kp}$ or not gives
\begin{align}
\notag&|Q_{\delta,k}^{(r)}f(x_1,\cdots,x_p)-Q_{\delta,k}^{(\infty)}f(x_1,\cdots,x_p)|\\
\notag\le &\int_{(0,\infty)^{kp}}|F_{r}(x_1,\cdots,x_{(k+1)p})-F_{\infty}
(x_1,\cdots,x_{(k+1)p})|\prod_{\ell=p+1}^{(k+1)p}\dd x_\ell\\
\label{eq:2019_2}\le &\frac{2Ckp}{\gamma^{kp}} e^{-\gamma M}+\varepsilon_{r,M}M^{kp},
\end{align}
where 
$$\varepsilon_{r,M}:=\sup_{(x_1,\cdots,x_{(k+1)p})\in [0,M]^{(k+1)p}}|
F_{r}(x_1,\cdots,x_{(k+1)p})-F_{\infty}(x_1,\cdots,x_{(k+1)p})|$$
converges to $0$ as $r\to\infty$, with $M$ fixed. 
Combining \eqref{eq:2019_1} and \eqref{eq:2019_2} gives
$$\|Q_{\delta,k}^{(r)}-Q_{\delta,k}^{(\infty)}\|\le \frac{2Ckp}{\gamma^{kp}} e^{-\gamma M}+\varepsilon_{r,M}M^{kp}.$$
On letting $r\rightarrow\infty$ followed by $M\rightarrow\infty$ gives
that $\|Q_{\delta,k}^{(r)}-Q_{\delta,k}^{(\infty)}\|$ converges to $0$ as $r\rightarrow\infty$, 
and so  $\lambda(Q_{\delta,k}^{(r)})$ converges to $\lambda(Q_{\delta,k}^{(\infty)})$ by Lemma \ref{lem:conti}. This completes the proof of part (c).
\end{proof}

\begin{proof}[Proof of Lemma \ref{lem:k+}]
Let $\delta>0$ be fixed as in the proof of Theorem \ref{thm:ar+}. Recall that we have
$\phi(t)\le C_2 e^{-\delta_2 |t|}$, where $\delta_2>p\delta$. 
Choose $\rho\in (0,1)$ such that $\sum_{j=1}^p|a_j|\rho^{-j}=1$. 
Fix $\varepsilon>0$ such that $p\delta(1+\varepsilon)<\delta_2$ and $\varepsilon<p \delta(1-\sum_{j=1}^p|a_j|)$.  For any positive integer $k$ (the exact choice of which is specified below in \eqref{eq:choose_k}), define the variables $(A_1,\cdots,A_{kp})$ inductively by setting
\begin{eqnarray*}
A_\ell-\sum_{j=1}^{\ell-1} |a_j|A_{\ell-j}&=&\delta(1+\varepsilon)\text{ if }1\le \ell\le p,\\
A_\ell-\sum_{j=1}^p |a_j|A_{\ell-j}&=&\min(\rho^\ell,\varepsilon)\text{ if }p+1\le \ell\le kp.
\end{eqnarray*}
Note that $A_\ell>0$ for all $1\le \ell\le kp$ by definition.
We now claim that 
\begin{align}\label{claim1}
\max_{1\le \ell \le kp}A_\ell<& \delta_2\text{ for all }k\ge 1,\\
\label{claim2}\max_{kp+1\le \ell \le (k+1)p}\sum_{j=\ell-kp}^p |a_j|A_{\ell-j}\le &(K+kp+p)\rho^{kp+1},\text{ for all  }k\ge 1,
\end{align} 
where $K:=p\delta(1+\varepsilon)\rho^{-p}$.
Note that \eqref{claim2} implies in particular that there exists $k$ such that 
\begin{align}\label{eq:choose_k}
\max_{kp+1\le \ell \le (k+1)p}\sum_{j=\ell-kp}^p |a_j|A_{\ell-j}<\delta.
\end{align}
Deferring the proof of \eqref{claim1} and \eqref{claim2} we will fix such a $k$, and finish the proof of the lemma. To this end, invoking \eqref{claim1} and using the bound 
$\phi(t)\le C_2e^{-A_\ell |t|}$ for all $t\in \R$ and 
$1\le \ell\le kp$ we have
\begin{align*}
F(x_1,\cdots,x_{(k+1)p}) 
\le &C_2^{kp}e^{\delta\sum_{j=1}^p(x_{kp+j}-x_j)}\prod_{\ell=p+1}^{(k+1)p}e^{-A_{(k+1)p+1-\ell}|x_\ell-\sum_{j=1}^p a_j x_{\ell-j}|}\\
\le &C_2^{kp}e^{\delta\sum_{j=1}^p(x_{kp+j}-x_j)}\prod_{\ell=p+1}^{(k+1)p}e^{-A_{(k+1)p+1-\ell}\Big( x_\ell-\sum_{j=1}^p |a_j| x_{\ell-j}\Big)}\\
=&C_2^{kp}\prod_{\ell=1}^{(k+1)p}e^{-\alpha_{(k+1)p+1-\ell} x_\ell},
\end{align*}
where 
$$
\alpha_\ell:=
\begin{cases} A_\ell-\sum_{j=1}^{\ell-1}A_{\ell-j}|a_j|-\delta&\text{if }1\le \ell\le p,
\\
A_\ell-\sum_{j=1}^pA_{\ell-j}|a_j|&\text{if }p+1\le \ell\le kp,
\\
\delta-\sum_{j=\ell-kp}^p|a_j|A_{\ell-j} & \text{if }kp+1\le \ell\le (k+1)p.\end{cases}
$$

By the choice of $(A_1,\cdots,A_{kp})$ we have  $\alpha_\ell>0$ for $1\le \ell\le (k+1)p$, which on setting $C:=C_2^{kp}$ and $\gamma:=\min_{1\le \ell\le (k+1)p}\alpha_\ell>0$ gives the desired conclusion. The uniformity of the choice of $C$ is immediate, and for the uniformity of $\gamma$, note that $$\gamma\ge \min\Big(\delta \varepsilon, \rho^\ell, \delta-(K+kp+p) \rho^{kp+1}\Big),$$
and these parameters are uniform over the parameter space $\sum_{j=1}^pa_j\le 1-\eta$ for any $\eta>0$.
\\

It thus remains to prove \eqref{claim1} and \eqref{claim2}, which we break down into a few steps.

First we show by induction on $\ell$ that
 \begin{equation}
   \label{eq-Aof}
   A_\ell\le \ell\delta(1+\varepsilon)
   \quad
   \mbox{\rm for $1\leq \ell\le p$.}
 \end{equation}
Indeed,  for $\ell=1$ we have $A_1=\delta(1+\varepsilon)$ by definition.
 If \eqref{eq-Aof}
holds for $1\le j\le \ell-1$, 
then using the formula for $A_\ell$ gives
\begin{align*}
  A_\ell&=\delta(1+\varepsilon)+\sum_{j=1}^{\ell-1}A_{\ell-j}|a_j|\le\delta(1+\varepsilon)+\max_{1\le j\le \ell-1}A_{\ell-j}\\
  &\le \delta(1+\varepsilon)+(\ell-1)\delta(1+\varepsilon)\le\ell \delta(1+\varepsilon),\end{align*}
which completes the induction.
 \\
 
 Next we show that
\begin{equation}
  \label{eq-Apof}
  A_\ell\le \delta p(1+\varepsilon) \quad 
  \mbox{\rm for all $1\le \ell\le kp$}.
\end{equation}
Indeeed, note that \eqref{eq-Apof} follows from
\eqref{eq-Aof} for $\ell\leq p$.
For larger $\ell$, use the definition of $A_\ell$ along with 
induction to note that, by the choice of $\varepsilon$,
 $$A_\ell\le \varepsilon+ (\sum_{j=1}^p|a_j|)\max_{1\le j\le p}A_{\ell-j}\le\varepsilon+ p\delta(1+\varepsilon)\sum_{j=1}^p|a_j| \le p\delta(1+\varepsilon).$$
 Note that \eqref{eq-Aof} and \eqref{eq-Apof} together yield
 \eqref{claim1}, since
  $p\delta(1+\varepsilon)<\delta_2$.

 We turn to the proof of \eqref{claim2} for which we first show that
 \begin{equation}
  \label{eq-AKof}
  A_\ell\le (K+\ell)\rho^\ell \quad \mbox{\rm for all $1\le \ell\le kp$},
\end{equation}
where $K=p\delta(1+\varepsilon)\rho^{-p}$ as before. 
The choice of $K$ and the observation that $A_\ell\le \ell \delta(1+\varepsilon)$ 
yields \eqref{eq-AKof}  for $1\le \ell\le p$.
For $\ell>p$, we proceed by induction.
  Assume \eqref{eq-AKof} holds for all 
  $1\le \ell'\le \ell-1$,
 and note that
\begin{align*}
A_\ell\le \rho^\ell+\sum_{j=1}^p|a_j|A_{\ell-j}
\le &\rho^\ell+\sum_{j=1}^p(K+\ell-j)|a_j|\rho^{\ell-j}\\
\le &\rho^\ell\Big[1+(K+\ell-1)\sum_{j=1}^p|a_j|\rho^{-j}\Big]=(K+\ell)\rho^\ell.
\end{align*}
This yields \eqref{eq-AKof}, which for $\ell\in [kp+1,(k+1)p]$ gives
$$\sum_{j=\ell-kp}^p|a_j|A_{\ell-j}\le (K+\ell)\rho^\ell\sum_{j=1}^p|a_j|\rho^{-j}=(K+\ell)\rho^\ell,$$
from which \eqref{claim2} follows trivially. 
\end{proof}

For the proof of Theorem \ref{thm:strict_ar}, we need the following lemma.
\begin{lem}\label{lem:logconcave}
  Suppose $\xi$ has a strictly positive log-concave
  density $\phi$ with respect to the Lebesgue measure on $\R$.
   Then, for any $\delta\geq 0$ and any bounded non-decreasing function $g$ on $\R$ ,
$$
\E[ g(\xi+\delta) | \xi +\delta>0] = \frac{ \E[ g(\xi+\delta) 1_{\{\xi+\delta>0\}} ] }{ \P( \xi+\delta > 0) } \geq \frac{ \E[ g(\xi) 1_{\{\xi>0\}} ] }{ \P( \xi > 0) } = \E[ g(\xi) | \xi>0].
$$
\end{lem}
We note that
one can construct even unimodal densities for which the conclusion of
Lemma~\ref{lem:logconcave} is false.

 \begin{proof}
The lemma follows from the inequality in \cite[Theorem~3]{preston1974}: one uses 
\begin{align*}
f_1(x)&:=(\int_0^\infty \phi(t) \dd t)^{-1} \phi(x) \one_{[0,\infty)}(x),\\
f_2(x)&:=(\int_{-\delta}^\infty \phi(t) \dd t)^{-1} \phi(x-\delta) \one_{[0,\infty)}(x).
\end{align*}
 One can check easily that the log concavity of $\phi$ implies the assumptions for $f_1, f_2$ required by \cite{preston1974}.
\end{proof}

We further remark that for finite subsets of $\R$, Lemma \ref{lem:logconcave} follows from Holley's inequality for finite lattices \cite{holley1974}.

\begin{proof}[Proof of Theorem \ref{thm:strict_ar}]
For proving strict monotonicity we will invoke Lemma \ref{lem:monotone} with $P=P^{{\bf b}}$ and $Q=P^{{\bf a}}$ and
$$h({\bf x}):=\frac{\P_{\xi}(\xi+\sum_{j=1}^p b_j x_{p+1-j}\geq
0)}{\P_\xi(\xi+\sum_{j=1}^p a_jx_{p+1-j}\geq 0)}.$$
For any $g\in \cB_{+,>}(S)$, we have
by Lemma \ref{lem:logconcave},
\begin{eqnarray*}
  \frac{P_S(g)}{Q_S(g)}&=&\frac{\E_{\xi} g(x_2,\cdots,x_p,\xi+\sum_{j=1}^p b_j x_{p+1-j}) 1_{\{\xi+\sum_{j=1}^p b_jx_{p+1-j}\geq 0\}}}{\E_{\xi} g(x_2,\cdots,x_p,\xi+\sum_{j=1}^p a_j x_{p+1-j}) 1_{\{\xi+\sum_{j=1}^p a_jx_{p+1-j}
  \geq 0\}}}
\\
&\ge &\frac{\P_\xi(\xi+\sum_{j=1}^p b_j x_{p+1-j}
\geq 0)}{\P_\xi(\xi+\sum_{j=1}^p a_j x_{p+1-j}\geq 0)} = h({\bf x}),
\end{eqnarray*}
showing that condition (i) holds. Proceeding to checking
 condition (ii), for any $g\in \cB_{+,>}(S)$ we have 
\begin{align*}
[P_{S}(g)]({\bf x})
=\P_\xi(\xi+\sum_{j=1}^p b_jx_{p+1-j}>0)[\widetilde{P}(g)]({\bf x}),
\end{align*}
where 
$$[\widetilde{P}(g)]({\bf x}):=\frac{\int_{y+\sum_{j=1}^p b_j x_{p+1-j}
\geq 0} g(x_2,\cdots,x_p,y+\sum_{j=1}^{p}b_j x_{p+1-j})\dd F(y)}{\int_{y+\sum_{j=1}^p b_j x_{p+1-j}
\geq 0} \dd F(y)}.$$
Since ${\bf b}\ge {\bf 0}$ we have that 
${\bf x}\mapsto \P(\xi+\sum_{j=1}^p b_j x_{p+1-j}\geq 0)$ is non-decreasing in ${\bf x}$, and so it suffices to show that $\widetilde{P}$ is  non-decreasing. To this end, for any $g\in \cB_{+,>}(S)$ and ${\bf x},{\bf y}\in S$ with ${\bf x}\le {\bf y}$, write
  $\xi_{\bf x}=\xi+\sum_{j=1}^p b_j x_{p+1-j}$ and
$\xi_{\bf y}=\xi+\sum_{j=1}^p b_j y_{p+1-j}$.
Then,
\begin{align*}
\frac{[\widetilde{P}(g)]({\bf y})}{[\widetilde{P}(g)]({\bf x})}
=&
  \frac{\E_{\xi} \left(g(y_2,\cdots,y_p,\xi_{\bf y}) 
  1_{\{\xi_{\bf y}\geq0\}} \,|\,\xi_{\bf y}
  \geq0\right)
}{\E_{\xi} \left(
g(x_2,\cdots,x_p,\xi_{\bf x}) 1_{\{\xi_{\bf x}\geq0\}}
\,|\,\xi_{\bf x}\geq0\right) }\\
\ge &
 \frac{\E_{\xi} \left(g(x_2,\cdots,x_p,\xi_{\bf y}) 
1_{\{\xi_{\bf y}\geq0\}} \,|\,\xi_{\bf y}\geq0\right)
}{\E_{\xi} \left(
g(x_2,\cdots,x_p,\xi_{\bf x}) 1_{\{\xi_{\bf x}\geq0\}}
\,|\,\xi_{\bf x}\geq0\right) }
\ge 1,
\end{align*}
where the first inequality uses the fact that $g$ is 
coordinate-wise increasing, and the second inequality uses Lemma 
\ref{lem:logconcave} and the positivity of ${\bf b}$.

Having verified that its conditions are satisfied,
we apply Lemma~\ref{lem:monotone} and get 
$[P_S^{n-p}({\bf 1})]({\bf x})
\ge [Q_{S,h}^{n-p}({\bf 1})]({\bf x})$, which, setting {\small $A_{\bf x}:=\{(Z_0,\ldots,Z_{p-1})={\bf x}\}$}, is the same as
$$\P^{\bf b}(\min_{p\le i\le n}Z_i>0|A_{\bf x})
\ge \E^{\bf a}\Big[\prod_{i=p}^n 1_{\{Z_i>0\}} h(Z_{i-p},\cdots,Z_{i-1})|
A_{\bf x}\Big].$$
Let ${\cal Z}_+=\{Z_i\geq 0, 0\leq i\leq n\}$.
Multiplying both sides of the last display
by $1_{\{\min_{0\le j\le p-1}Z_j>0\}}$,
taking expectations with respect to $\mu$ and rearranging gives
$$\P^{\bf b}({\cal Z}_+)\ge 
\P^{\bf a}( {\cal Z}_+ )\E^{\bf a}\Big[\prod_{i=p}^{n}h(Z_{i-p},\cdots,Z_{i-1})|{\cal Z}_+\Big].$$
By Proposition \ref{ppn:finite} we have $ \theta_F({\bf a})>0$, and so, given the above inequality, it suffices to show that
\begin{align}\label{eq:show}\liminf_{n\rightarrow\infty}\frac{1}{n}\log \E^{\bf a}\Big[\prod_{i=p}^{n}h(Z_{i-p},\cdots,Z_{i-1})|
   {\cal Z}_+\Big]>0.
\end{align}
For showing \eqref{eq:show}, 
we claim the existence of $k>1$ such that
\begin{align}\label{eq:show2}
\limsup_{n\to\infty} \frac{1}{n}\log \P^{\bf a}\Big({L}_n[k,\infty)\ge \frac{1}{4p}|
   {\cal Z}_+\Big)<0,\\
\limsup_{n\to\infty} \frac{1}{n}\log \P^{\bf a}\Big({L}_n[0,1/k]\ge \frac{1}{4p}|
 {\cal Z_+}\Big)<0 \label{eq:show3},
\end{align}
where $L_n^{\bf Z}:=\frac{1}{n}\sum_{i=p}^n \delta_{
  Z_i}$
is an empirical measure of total mass $\frac{n-p+1}{n}$.
Indeed, given \eqref{eq:show2} and \eqref{eq:show3}  we have 
\begin{align*}&\E^{\bf a}\Big[\prod_{i=p}^{n}h(Z_{i-p},\cdots,Z_{i-1})|
 {\cal Z}_+\Big]\\
 &\quad \ge (1+\eta)^{\frac{n}{2p}-o(1)}\P^{\bf a}\Big(L_n[1/k,k]\ge \frac{n-p+1}{n}-\frac{1}{2p}|
{ \cal Z_+}\Big)
\end{align*}
with $\eta:=-1+\inf_{ {\bf x}\in [1/k,k]^ph({\bf x})}$. Also we have $\eta>0$, since $\xi$ has a strictly positive density on the whole of $\R$, 
 which implies  that the continuous function $h$ is strictly greater than $1$ point wise on the compact set $[1/k,k]^p$.
   \eqref{eq:show} follows after applying $\log$, normalizing by $n$ 
 and taking limits.
\\
It thus remains to prove \eqref{eq:show2} and \eqref{eq:show3}. For this, recall that the initial distribution $\mu$ satisfies, for any 
$\lambda'\in (0,\delta)$,
\begin{align}\label{eq:exp_moment_finite}
\frac{1}{n}\log  \E_\mu 
\left(e^{\lambda_1'\sum_{j=0}^{p-1}Z_j}1_{\{\min_{0\le i\le p-1}Z_i>0\}}\right)
<\infty.
\end{align}
Proceeding to showing \eqref{eq:show2},
by the log concavity of $\phi$ there exist
$\lambda_0>0$, $\lambda_1\in (0,\delta)$ 
such that $\log \phi(x)\le \lambda_0-\lambda_1 |x|$ for all $x\in\R$, and so with $$\widetilde{L}:=\Big\{{\bf x}=(x_p,\cdots,x_n)\in [0,\infty)^{n-p+1}:|i\in [p,n]:x_i\ge k|\ge \frac{n}{4p}\Big\}$$ we have
\begin{eqnarray*}
  \P^{\bf a}\Big({L}_n[k,\infty)\ge \frac{1}{4p},
   {\cal Z}_+\Big|Z_0,\cdots,Z_{p-1}\Big)
&\le &\int\limits_{\widetilde{L}
}\prod_{i=p}^n\phi(x_i-\sum_{j=0}^pa_jx_{i-j})\dd x_i \\
&\le &\!\!\!\!e^{n\lambda_0} \!
  \int\limits_{\widetilde{L}}
    \prod_{i=p}^ne^{-\lambda_1|x_i-\sum_{j=0}^pa_jx_{i-j}|}\dd x_i
\\
&\le &\!\!\!\!e^{n\lambda_0}\!
\int\limits_{\widetilde{L}}
    \prod_{i=p}^ne^{-\lambda_1x_i+\lambda_1\sum_{j=0}^pa_jx_{i-j}}\dd x_i
\\
&\le &\!\!\!\!e^{n\lambda_0+\lambda_1'\sum_{j=0}^{p-1}x_j}
  \int\limits_{\widetilde{L}}\prod_{i=p}^ne^{-\widetilde{\lambda}_1x_i}\dd x_i,
\end{eqnarray*}
where $\lambda_1':=\lambda_1(\sum_{j=1}^p a_j)<\delta$, and $\widetilde{\lambda}_1:=\lambda_1-\lambda_1'>0$. 
Integrating both sides with respect to $1\{x_i\ge 0,0\le i\le p-1\}d\mu(x_0,\cdots,x_{p-1})$ gives 
\begin{align*}
&\P^{\bf a}\Big({L}_n[k,\infty)\ge \frac{1}{4p},
   {\cal Z}_+\Big)\\
   &\quad \le e^{n\lambda_0}\E e^{\lambda_1'\sum_{j=0}^{p-1}Z_j}
1_{\{\min_{0\le j\le p-1}Z_j>0\}}\widetilde{\lambda_1}^{-n}\P\Big(L_n^{\bf Y}[k,\infty)\ge \frac{1}{4p}\Big),
\end{align*}
where $(Y_i,p\le i\le n)$ are i.i.d.\ exponential random variables with parameter $\widetilde{\lambda}_1$, and $L_n^{\bf Y}:=\frac{1}{n}\sum_{i=p}^n\delta_{Y_i}$. Since $\E e^{\lambda_1'\sum_{j=0}^{p-1}Z_j}1_{\{\min_{0\le j\le p-1}Z_j>0\}}<\infty$ by \eqref{eq:exp_moment_finite}, it suffices to show that
\begin{align*}
\limsup_{k\rightarrow\infty}\limsup_{n\rightarrow\infty}\frac{1}{n}\log\P\Big(L_n^{\bf Y}[k,\infty)\ge \frac{1}{4p}\Big)=-\infty.
\end{align*}
But this follows on invoking Sanov's theorem to note that $L_n^{\bf Y}$ satisfies a large deviation principle with a good rate function on $M_1(0,\infty)$, the set of probability measures on $(0,\infty)$ with respect to the weak topology. 
Thus \eqref{eq:show2} holds, and a similar proof shows \eqref{eq:show3}. 
\end{proof}

\begin{proof}[Proof of Proposition \ref{thm:non_strict}]


Fix $\varepsilon>0$ such that $\P( (Z_0,\ldots, Z_{p-1}) \in (\varepsilon,\infty)^p )>0$. This exists by the choice of the initial distribution. Define an associated AR process $\{Z_i'\}_{i\ge 1}$ on the same probability space by setting
$$ 
Z_i' := \begin{cases}
  0 & i\in \{0,\ldots, p-2\}\\ \varepsilon & i=p-1\\
  \sum_{j=1}^p a_j Z_{i-j}' + \xi_i & i\geq p.
\end{cases}
$$ 
Since $a_i\geq 0$, on the set $\{ \min_{0\le i\le p-1} Z_i> \varepsilon\}$ we have $Z_n\geq Z_n'$ for all $n\geq p$. Thus,
\begin{eqnarray*}
  \P( \min_{0\le i\le n}Z_i \geq 0 ) &\geq & \P( \min_{0\le i \le p-1}Z_i>\varepsilon, \min_{p\le i\le n} Z_i > 0)
\\
& \geq& \P( \min_{0\le i\le p-1}Z_i> \varepsilon,\min_{p\le i\le n} Z_i' > 0)\\
& =& \P( \min_{0\le i\le p-1}Z_i>\varepsilon) \cdot \P( \min_{p\le i\le n}Z_i' > 0)
\end{eqnarray*}
and so it suffices to show that
\begin{align}\label{eq:review}
\limsup_{n\rightarrow\infty}\frac{1}{n}\log\P( \min_{p\le i\le n}Z_i' > 0)=0.
\end{align}
To this effect, use induction to note that
\begin{equation} \label{eqn:arrepr}
Z'_{i+p-1}=\sum_{j=0}^i b_{i-j} \xi_j
\end{equation}
where 
$\xi_0=\varepsilon$, $b_0:=1$ and $b_i:=\sum_{j=1}^p b_{i-j} a_{j}$ and $b_i=0$ for $i<0$.
Also, we claim that for any ${\bf a}\ge {\bf 0}$ such that $\sum_{j=1}^p a_j>1$ there exists $\alpha,\beta>0$ and a (unique) $\rho>1$ such that
\begin{equation}  \label{eqn:boundsonb}
 \alpha \rho^i \leq b_i \leq \beta \rho^i,\qquad i\ge 1.
\end{equation}
To see this, consider the function $f(\rho):=\sum_{j=1}^p a_j \rho^{-j}$. This function is strictly decreasing on $[1,\infty)$ and satisfies
  $f(1)=\sum_{j=1}^p a_j >1$ and $f(\infty)=0$. Therefore, there must be a (unique) $\rho>1$ with $f(\rho)=1$.
Using the definition of $\{b_j\}_{j\in \Z}$, it is easy to see that for this $\rho$ one can find $\alpha,\beta>0$ satisfying (\ref{eqn:boundsonb}).
Proceeding to verify \eqref{eq:review}, set
$$
 {\cal S}_n:=\bigcap_{i=1}^n \{ \xi_i > - \min(\kappa \rho^i i^{-2},M)\},
$$
with some fixed $0<\kappa < \varepsilon \alpha ( \beta \sum_{k=1}^\infty k^{-2})^{-1}$, and $M>0$.
We will show that on the set ${\cal S}_n$ we have $\min_{0\le i\le n+p-1}Z'_{i}>0$ for any fixed $M$. Given this, since $ \P(
 {\cal S}_n)\ge \P(\xi_1>-M)^n$ and $M$ is arbitrary, \eqref{eq:review} follows. 

It remains to show that persistence happens on $ {\cal S}_n$, for which use \eqref{eqn:arrepr} and \eqref{eqn:boundsonb} to note that
\begin{eqnarray*}
Z'_i&=& \sum_{j=0}^{i+p-1} b_{i-j+p-1} \xi_j
=  b_{i+p-1} \xi_0 + \sum_{j=1}^{i+p-1} b_{i+p-1-j} \xi_j \\
&\geq &\alpha \rho^{i+p-1} \varepsilon + \sum_{j=1}^{i+p-1} - \beta \rho^{i+p-1-j} \kappa \rho^j j^{-2}
\geq  \rho^{i} \rho^{p-1}(\varepsilon \alpha - \beta \kappa \sum_{j=1}^\infty j^{-2}) > 0,
\end{eqnarray*}
by the choice of $\kappa$, and so the proof is complete.

\end{proof}

\begin{proof}[Proof of Proposition \ref{ppn:finite}] 

 Since $0$ is in the interior of the support of $F$, there exists $\alpha<0,\beta>0$ such that $[\alpha,\beta]$ is contained in the support. Then for any $(a,b)\in (\alpha,\beta)$ we have $\P_F(\xi_1\in (a,b))>0$. Indeed, otherwise the complement of $(a,b)$ is a closed set of probability $1$, which implies $(a,b)$ is not in the support of $F$, a contradiction. 

 Let $N:=\max(1,\lceil\sum_{j=1}^p|a_j|\rceil)$ be a positive integer, and set $$L:=\min(\beta/N,-\alpha/N)>0.$$
Then we claim that $\P(Z_i\in (0,L),0\le i\le n)$ decays slower than some exponential rate. Indeed, setting $Y_n:=\sum_{j=1}^p a_j Z_{n-j}$ on this set, we have $|Y_n|< NL$. Now setting $I_k:=(-LN+\frac{(k-1)L}{2},-LN+\frac{kL}{2})$ for $k\in \{1,\ldots,4N\}$ note that if $Y_n\in I_k$, then on the set $\{Z_i\in [0,L),0\le i\le n-1\}$ we have
$$\P(Y_n+\xi_n\in[0,L)|Z_0,\cdots,Z_{n-1})\ge \P(\xi_n\in J_k),
$$
where $J_k:=(LN-\frac{(k-1)L}{2},LN-\frac{(k-2)L}{2})$.  Also note that $J_k\subset (\alpha,\beta)$ for any $k\in\{1,\ldots,4N\}$ by the
choice of $L$, as $\beta\ge LN$ and $\alpha\le -LN$. An inductive argument then
gives
$$\P(\cap_{i=0}^n \{Z_i\in [0,L)\})\ge \P(
  \cap_{i=0}^{p-1} \{Z_i\in[0,L)\})\Big(\min_{1\le k\le N}\P(\xi_1\in J_k)\Big)^{n-p},$$
from which the desired conclusion follows.
\end{proof}

\section{Proofs for the exponents in the concrete examples}\label{sec:comp-proof} 
We only give hints on how to solve the concrete eigenvalue equations.

\begin{proof}[Proof of Proposition \ref{ppn:example_ma_uniform}]
Recall from Theorem~\ref{thm:mainfty} that the eigenvalue equation in this case is $\lambda g(x) = \E[ g(\xi) 1_{\{\xi>-x\}}]$, $x\in\R$.

Let us start with some facts for distributions with bounded support: Assume $\xi\in(-a,b)$ almost surely. Then one can show that $g(x)=0$ for $x\leq -b$. Further, $g$ is constant on $[a,\infty)$. If $a>b$ then $g$ is constant even on $[b,\infty)$. If $a\leq b$ one finds that $g(x)=\lambda^{-1} g(a) \P( \xi > - x)$ for $x\in (-b,-a]$. 

Now assume that $\xi$ is uniformly distributed in $(-a,b)$. In this case, the eigenvalue equation becomes
\begin{equation} \label{eqn:concrete1}
\lambda g(x) = \int_{-x}^b g(y) \dd y \frac{1}{a+b}, \qquad x\in(-a,b).
\end{equation}
We split in cases.

(a) $a\le b$: In this case, we already know $g$ on $(-a,a)^c$ from the above observations.
For the range $x\in(-a,a)$, the functions $g(x) = \kappa(\cos( \alpha x) + \sin(\alpha x))$ can be seen to be the only solutions (e.g.\ by differentiating (\ref{eqn:concrete1}) twice), where necessarily $\alpha=\frac{1}{(a+b)\lambda}$. The restrictions of the integral equation are equivalent to (\ref{eqn:uniformmasoln}).
Since $\alpha=\frac{1}{(a+b)\lambda}$ and $a,b$ are known, this is a non-linear equation for $\lambda$. It has several solutions, but we are interested in the smallest possible value for $\lambda^{-1}$, which corresponds to the unique non-negative eigenfunction, as one can check.


%
%
%
%

(b) $a\geq b$: This case is actually simpler. We already know from the observations on distributions with bounded support that $g$ is zero left of $-b$ and constant right of $b$. Thus, it only remains to consider $x\in(-b,b)$. One can check that the functions $g(x) = \kappa ( \cos(\alpha x) + \sin(\alpha x))$ are the only solutions to the integral equation for $x\in(-b,b)$ with $\alpha=1/(\lambda(a+b))$ (e.g.\ by differentiating (\ref{eqn:concrete1}) twice). The restrictions from the integral equation are equivalent to $\cos(\alpha b) = \sin(\alpha b)$, which holds for $\alpha_k b = \pi/4+k\pi$, $k\in\Z$. Since we are interested in the largest possible value for $\lambda$ (which corresponds to the unique non-negative eigenfunction), the solution in this case is $\alpha_0 b = \pi/4$.
\end{proof}



\begin{proof}[Proof of Theorem~\ref{thm:example_ma_general_at_rho1}]
\fa{
The first observation is that the persistence probability is in fact distribution-free (for symmetric densities). This can be seen by representing the i.i.d.\ $(\xi_i)$ with the help of i.i.d.\ $(U_i)$ that are {\it uniform} on $[0,1]$: $\xi_i = F^{-1}(U_i)$, where $F$ is the distribution function of $\xi_1$. The fact that the $(\xi_i)$ has a symmetric density is equivalent to $-F^{-1}(u)=F^{-1}(1-u)$ for all $u\in[0,1]$. This shows
\begin{eqnarray*}
\P(\min_{1\leq i\leq n} Z_i \geq 0) 
&=& \P(\forall i \in\{1,\ldots,n\}  : \xi_i \ge -\xi_{i-1})
\\
& =& \P(\forall i \in\{1,\ldots,n\}  : F^{-1}(U_i) \ge -F^{-1}(U_{i-1}))
\\
& =& \P(\forall i \in\{1,\ldots,n\}  : F^{-1}(U_i) \ge F^{-1}(1-U_{i-1}))
\\
& =& \P(\forall i \in\{1,\ldots,n\}  : U_i \ge 1-U_{i-1})
\\
& =& \P(\forall i \in\{1,\ldots,n\}  : 2(U_i-1/2) + 2(U_{i-1}-1/2) \ge 0).
\end{eqnarray*}
This shows that the persistence probability is distribution free and the problem can be reduced to the density $\phi=\frac{1}{2}\one_{[-1,1]}$.
}

\fa{
In this case, the eigenvalue equation reads
$$
\lambda g(x) =[Kg](x):= \int_{-x}^1 g(z) \,\frac{ \dd z}{2}, \qquad x\in[-1,1].
$$
Differentiating twice yields $g''(x)=-(2\lambda)^{-2} g(x)$, whose only solutions (up to constant multiples) are of the form $g(x)=\sin( x / (2\lambda)) + \cos( x / (2\lambda))$. In order to satisfy not only the differential equation but also the original eigenvalue equation, one needs to demand that $g(-1)=0$, which gives the set of solutions $\lambda_k^{-1} = \pi/2 + 2\pi k$, $k\in\Z$. The positive eigenfunction (or equivalently the largest eigenvalue) is obtained for $k=0$. This identifies the persistence exponent.
}

\fa{
To see the explicit formula in the statement one easily sees that $(g_k)$ is an orthonormal basis of $L_2[-1,1]$, $\one=\sum_{k\in\Z} (-1)^k \sqrt{2} \lambda_k g_k$, $\E[g_k(\xi_0)]= (-1)^k \sqrt{2} \lambda_k$, and that the persistence probability equals $\E[ K^n \one (\xi_0)]$, the latter of which can then be readily computed.
}
\end{proof}

\begin{proof}[Proof of Proposition~\ref{ppn:example_ma_exponential_allq}]
By Theorem~\ref{thm:mainfty}, the eigenvalue equation reads:
$$
\lambda g(x) = \int_{\fa{\max(0,-a_1x)}}^\infty g(y) e^{-y} \dd y,\qquad x\in\R.
$$
One checks easily that $g(x):=e^{a_1 x /(1+a_1)} 1_{x\ge 0} + 1_{x\le 0}$ is a non-negative eigenfunction for the eigenvalue $\lambda=1+a_1$.
Note however that one needs to verify that $\lambda=1+a_1$ is the {\it largest} eigenvalue of this operator. To this effect, let $\beta\ge \lambda>0$ denote the largest eigenvalue of $K$, and use Theorem~\ref{thm:exponent} to get the existence of a non-negative, finitely additive measure $m$ on $\R$ such that $m(\R)=1$ and for every $\psi\in \cC_b(\R)$ we have
\begin{align}\label{eq:ma_exp}
\int_{-\infty}^\infty m(\dd x) \int_{\max(0,-a_1 x)} e^{-y} \psi(y) \dd y=\beta m(\psi).
\end{align}
Thus, setting $\psi=g$ in \eqref{eq:ma_exp} gives
$$\beta m(g)=\int_{-\infty}^\infty m(\dd x) \int_{\max(0, -a_1 x)} e^{-y} g(y) \dd y=\int_{-\infty}^\infty Kg(x)m(\dd x)=\lambda m(g),$$
where the last equality uses $Kg=\lambda g$. Thus to conclude $\beta=\lambda$, it suffices to show that $m(g)$ is not zero. {If $m(g)=0$, using the fact that $g$ is lower bounded by a positive constant on $(-\infty,L]$ gives $m(-\infty,L]=0$ for all $L\in\R$.
Consequently by finite additivity $m(L,\infty)=1$.} But invoking \eqref{eq:ma_exp} with $\psi(x)=1_{\{x>L\}}$ gives
$$\beta m(L,\infty)=\int_{-\infty}^\infty m(\dd x) \int_{y+a_1 x>0, y>L} e^{-y} \dd y\le \int_{-\infty}^\infty m(\dd x)\int_{y>L}e^{-y}\dd y=e^{-L},$$
which on taking limits as $L\rightarrow\infty$ gives $\beta=0$, which is a contradiction.
\end{proof}
%
\begin{proof}[Proof of Proposition \ref{ppn:ar_unif}]
Considering the eigenvalue equation in Theorem~\ref{thm:ar} gives the existence of a continuous non-negative function $g:[0,\infty)\mapsto [0,\infty)$ satisfying ($\lambda=\theta_F(-1)$ for short)
\begin{align*}
(a+b)\lambda g(x) = \int_x^b g(y- x) \dd y=\int_0^{b-x}g(y)\dd y, \qquad x \in(0,b],
\end{align*}
and $g(x)=0$ for $x>b$.
It is easy to check that the only solutions to this are given by multiples of $g(t)=\cos(\alpha_k t)$
with $\alpha_k b = \frac{\pi}{2} + \pi k$ for some $ k\in \mathbb{Z}$. This gives the corresponding eigenvalues 
$
\lambda_k=(-1)^k (\alpha_k(a+b))^{-1}
$
, the largest one of which is $\lambda_0$.
\end{proof}


\begin{proof}[Proof of Proposition~\ref{ppn:ar_exponential}]
%
One can check that $P_S \one = g$ with $g(x)=e^{a_1x}$. Therefore, the persistence probability can be computed as
$$
\E P_S^n \one (Z_0) 1_{\{Z_0>0\}} =\E P_S^{n-1} g (Z_0) 1_{\{Z_0>0\}}=\left( \frac{1}{1-a_1}\right)^{n-1} \E  g (Z_0) 1_{\{Z_0>0\}}.
$$
\end{proof}

{\bf Acknowledgement.} We thank the American Institute of Mathematics (AIM) for hosting  SQuaREs which brought the authors together and at which occasions the initial questions for this project were discussed. We thank Ohad Feldheim for help with the proof of Proposition \ref{thm:mafinite}, and Amir Dembo, Naomi Feldheim, and Fuchang Gao  for their helpful comments and suggestions. We would also like to thank the AE and an anonymous referee, whose comments have greatly improved the presentation of this paper.

\end{document}